\documentclass{amsart}
\usepackage{amssymb,latexsym}
\usepackage{amscd,amsthm}

\usepackage[all]{xy}
\usepackage{tikz}
\usetikzlibrary{matrix,arrows}
\usepackage{tikz-cd}
\usepackage{enumitem} % to label enumerated lists with Roman numerals

\newtheorem{theorem}{Theorem}[section]
\newtheorem{lemma}[theorem]{Lemma}

\newtheorem{proposition}[theorem]{Proposition}
\newtheorem{corollary}[theorem]{Corollary}

\theoremstyle{definition}
\newtheorem{definition}[theorem]{Definition}

\newtheorem*{remark}{Remark}

\newtheorem{example}[theorem]{Example}

\DeclareMathOperator{\Ext}{Ext}
\DeclareMathOperator{\Hom}{Hom}
\DeclareMathOperator{\Tor}{Tor}

\DeclareMathOperator{\cok}{cok}
\DeclareMathOperator{\im}{Im}

\newcommand{\cat}[1]{\mathcal{#1}}           %% font for categories

\newcommand{\tensor}{\otimes}

\newcommand{\class}[1]{\mathcal{#1}}   %% font for classes

\newcommand{\Z}{\mathbb{Z}}

\newcommand{\mathcolon}{\colon\,} %% Hovey uses for maps, like f: A -> B

\newcommand{\ch}{\textnormal{Ch}(R)}
\newcommand{\cha}[1]{\textnormal{Ch}(\mathcal{#1})}

\newcommand{\tilclass}[1]{\widetilde{\class{#1}}}
\newcommand{\dgclass}[1]{dg\widetilde{\class{#1}}}

\newcommand{\rightperp}[1]{#1^{\perp}}
\newcommand{\leftperp}[1]{{}^\perp #1}

\numberwithin{equation}{section} 

\begin{document}

\title[Canonical resolutions]{Canonical resolutions in hereditary abelian model categories}

\author{James Gillespie}
\address{Ramapo College of New Jersey \\
         School of Theoretical and Applied Science \\
         505 Ramapo Valley Road \\
         Mahwah, NJ 07430}
\email[Jim Gillespie]{jgillesp@ramapo.edu}
\urladdr{http://pages.ramapo.edu/~jgillesp/}

\date{\today}

\begin{abstract}
Each object of any abelian model category has a canonical resolution as described in this article. When the model structure is hereditary we show how morphism sets in the associated homotopy category may be realized as cohomology groups computed from these resolutions. We also give an alternative description of the morphism sets in terms of Yoneda Ext groups. 
\end{abstract}

\maketitle

\section{Introduction}\label{sec-intro}

In this paper we use techniques from classical homological algebra, and the more recent theory of cotorsion pairs, to define and study the natural idea of a canonical resolution. These are certain doubly infinite resolutions that exist in abelian model categories and we show how under natural assumptions they yield long exact cohomology sequences reminiscent of Tate cohomology sequences. 

To explain, let $\cat{A}$ be an abelian category. An abelian model structure on $\cat{A}$, a notion defined by Hovey in~\cite{hovey}, consists of a triple $\mathfrak{M} = (\class{Q},\class{W},\class{R})$ of classes of objects in $\cat{A}$ satisfying some homological conditions. One is that the class $\class{Q}$ of \emph{cofibrant} objects and the class $\class{R}_{\class{W}} := \class{W}\cap\class{R}$ of \emph{trivially fibrant} objects are required to be orthogonal with respect to Yoneda's functor $\Ext^1_{\cat{A}}(-,-)$. Similarly, the class $\class{Q}_{\class{W}} := \class{Q}\cap\class{W}$ of \emph{trivially cofibrant} objects is orthogonal to the class $\class{R}$ of \emph{fibrant} objects. This is expressed more precisely and succinctly  in Section~\ref{sec-prelims} by saying that  $(\class{Q},\class{R}_{\class{W}})$ and $(\class{Q}_{\class{W}},\class{R})$ are cotorsion pairs. Another important homological assumption on the triple $\mathfrak{M}$ is that these two cotorsion pairs are \emph{complete}. This too is defined precisely in Section~\ref{sec-prelims} but we point out now that it implies the existence of certain resolutions and coresolutions for each object in $\cat{A}$. In particular, given any object $A \in \class{A}$, we may construct resolutions as follows:

%For a trivial example, if $\cat{G}$ is a Grothendieck category and $\class{I}$ denotes the class of all injective objects, then $(\class{G},\class{G},\class{I})$ is determines a (trivial) model structure on $\class{G}$ and 

%To explain, suppose that $\mathfrak{M} = (\class{Q},\class{W},\class{R})$ is an hereditary abelian model structure on an abelian category $\cat{A}$. So $\class{Q}$ is the class of cofibrant objects, $\class{R}$ is the class of fibrant objects, $\class{W}$ is the class of trivial objects, and we have complete cotorsion pairs $(\class{Q}\cap\class{W}, \class{R})$ and $(\class{Q},\class{W}\cap\class{R})$. (Full definitions are in Section~\ref{sec-prelims}).  We will let $\class{Q}_{\class{W}} := \class{Q} \cap \class{W}$ denote the class of trivially cofibrant objects, and similarly $\class{R}_{\class{W}} := \class{W} \cap \class{R}$ denotes the class of trivially fibrant objects. 

\begin{enumerate}[label=(\roman*)]
\item Using completeness of the cotorsion pair  $(\class{Q}_{\class{W}},\class{R})$ we may construct an exact chain complex $$\class{Q}^A_{\circ} \twoheadrightarrow A \ \ \equiv \ \ \ \cdots \xrightarrow{} Q_2 \xrightarrow{d_2} Q_1 \xrightarrow{d_1} Q_0 \xrightarrow{\epsilon} A \xrightarrow{} 0 $$ so that each $Q_i \in \class{Q}_{\class{W}}$ and each kernel is in $\class{R}$.

\item Using completeness of the cotorsion pair $(\class{Q},\class{R}_{\class{W}})$ we may construct an exact chain complex $$A \hookrightarrow \class{R}_A^{\circ} \ \ \equiv \ \ \ 0 \xrightarrow{} A \xrightarrow{\eta} R_{-1} \xrightarrow{d_{-1}} R_{-2} \xrightarrow{d_{-2}} R_{-3} \xrightarrow{} \cdots $$ so that each $R_i \in \class{R}_{\class{W}}$ and each kernel is in $\class{Q}$.
\end{enumerate}
We may paste these together by setting $d_0 = \eta\epsilon$. Then we have $\ker{d_0} = \ker{\epsilon}\in \class{R}$, and,  $\cok{d_0} = \cok{\eta}\in \class{Q}$. Thus we obtain the following full resolution of $A$:
 $$W_A \ \ \equiv \ \ \ \cdots \xrightarrow{} Q_2 \xrightarrow{d_2} Q_1 \xrightarrow{d_1} Q_0\xrightarrow{d_0} R_{-1} \xrightarrow{d_{-1}} R_{-2} \xrightarrow{d_{-2}} R_{-3} \xrightarrow{} \cdots $$
 
 Using the notation and language of ``cochain complexes'' we could also denote this construction by 
 $$W^A \ \ \equiv \ \ \ \cdots \xrightarrow{} Q^{-3} \xrightarrow{} Q^{-2} \xrightarrow{d^{-2}} Q^{-1}\xrightarrow{d^{-1}} R^{0} \xrightarrow{d^0} R^{1} \xrightarrow{d^1} R^{2} \xrightarrow{} \cdots $$ where we have $\ker{d^{-1}} = \ker{\epsilon}\in \class{R}$, and,  $\cok{d^{-1}} = \cok{\eta}\in \class{Q}$. 
\begin{definition}
We call $W_A$ an  \textbf{$\mathfrak{M}$-resolution} of $A$. By a \textbf{canonical resolution} of $A$ we mean an $\mathfrak{M}$-resolution $W_{QA}$ of any cofibrant replacement $QA$ of $A$.
In the same way, we call $W^A$ an  \textbf{$\mathfrak{M}$-coresolution} of $A$ and by a \textbf{canonical coresolution} of $A$ we mean an $\mathfrak{M}$-coresolution $W^{RA}$ of any fibrant replacement $RA$ of $A$.
\end{definition}

Note that if $A$ is already cofibrant, then any $\mathfrak{M}$-resolution of $A$ is already a canonical resolution of $A$. The quintessential example of a canonical resolution is the complete projective resolution (i.e. totally acyclic complex of projective $R$-modules) associated to a Gorenstein projective approximation of an $R$-module $M$ over a suitable ring $R$. See Example~\ref{example-GorenProj} at the end of Section~\ref{sec-Ho-Ext}.

A concept which has proved to be of fundamental importance in the theory of abelian model structures is the \emph{hereditary} condition. This condition has been ever-present in the examples arising in applications; see~\cite{gillespie-hereditary-abelian-models}. It just means that $\textnormal{Ext}^i_{\cat{A}}(Q,R) = 0$ for all $i \geq 1$ whenever $Q$ is cofibrant (resp. trivially cofibrant) and $R$ is trivially fibrant (resp. fibrant).  The hereditary hypothesis is needed to guarantee that canonical resolutions possess the properties of Lemma~\ref{lemma-can-resolutions-cofibrant}. These properties will be crucial to us, so we  will also assume the hereditary condition throughout the entire paper. 

As we will see, both $\mathfrak{M}$-resolutions and canonical resolutions are unique up to chain homotopy. We use this in Sections \ref{sec-augmentations-homotopies} - \ref{sec-Ho-Ext} to introduce and study bifunctors denoted by $\Ext^n_{\textnormal{Ho}(\mathfrak{M})}(A,B)$, one for each integer value of $n$. To do so, we start by defining in Section~\ref{sec-canonical-resolutions}, see Definition~\ref{def-lExt}, a bifunctor $$\ell\Ext^n_{\mathfrak{M}}(A,B) := H^n[\Hom_{\cat{A}}(W_A,B)],$$ where $W_A$ may be any $\mathfrak{M}$-resolution of $A$. 
On the other hand, in Section~\ref{sec-cochain-complexes}, see Definition~\ref{def-rExt}, we consider the bifunctor $$r\Ext^n_{\mathfrak{M}}(A,B) := H^n[\Hom_{\cat{A}}(A, W^B)],$$ where $W^B$ is some $\mathfrak{M}$-coresolution of $B$. Theorem~\ref{them-lExt} and Theorem~\ref{them-rExt} summarize the main properties of these functors.
We see that they share nice properties when $A$ is cofibrant and $B$ is fibrant and in fact it is because we have a canonical isomorphism $\ell\Ext^n_{\mathfrak{M}}(A,B) \cong r\Ext^n_{\mathfrak{M}}(A,B)$ in this case. This is reminiscent of how the left and right homotopy relations coincide when the source is cofibrant and the target is fibrant. Because of the canonical isomorphism, when $A$ is cofibrant and $B$ is fibrant, we will simply denote this group by $\Ext^n_{\mathfrak{M}}(A,B)$, with the realization that it may be computed by either a canonical resolution of $A$, or, a canonical coresolution of $B$. 

We then get to the main point, and in Definition~\ref{def-Ho-Ext} we set, for each integer $n$,
$$\Ext^n_{\textnormal{Ho}(\mathfrak{M})}(A,B) := \Ext^n_{\mathfrak{M}}(RQA,RQB),$$
where $RQA$ and $RQB$ represent bifibrant (that is, cofibrant-fibrant) replacements.
The main properties of the functors $\Ext^n_{\textnormal{Ho}(\mathfrak{M})}$ are listed in Theorem~\ref{them-Ho-Ext}. It states first of all that there is a canonical isomorphism 
$$\Ext^n_{\textnormal{Ho}(\mathfrak{M})}(A,B) \cong \Ext^n_{\mathfrak{M}}(QA,RB)$$
showing that $\Ext^n_{\textnormal{Ho}(\mathfrak{M})}(A,B)$ may be computed by using any canonical resolution of $A$, or, any canonical coresolution of $B$. This result is reminiscent  of how morphism sets in $\textnormal{Ho}(\mathfrak{M})$ may be defined via  $\textnormal{Ho}(\mathfrak{M})(A,B) := \Hom_{\cat{A}}(RQA,RQB)/\sim$ (where $\sim$ represents the homotopy relation) but that there are canonical isomorphisms $\textnormal{Ho}(\mathfrak{M})(A,B) \cong \Hom_{\cat{A}}(QA,RB)/\sim$.

Second, Theorem~\ref{them-Ho-Ext} says that short exact sequences in $\cat{A}$ are sent to long exact sequences of $\Ext^n_{\textnormal{Ho}(\mathfrak{M})}$ groups. That is, suppose we have a short exact sequence $0 \xrightarrow{} A  \xrightarrow{} B  \xrightarrow{} C  \xrightarrow{} 0$ in $\cat{A}$. Then for any object $X$ we get a long exact cohomology sequence of abelian groups
$$\cdots  \xrightarrow{} \Ext^{n-1}_{\textnormal{Ho}(\mathfrak{M})}(X,C) \xrightarrow{}   \Ext^{n}_{\textnormal{Ho}(\mathfrak{M})}(X,A) \xrightarrow{}  \Ext^{n}_{\textnormal{Ho}(\mathfrak{M})}(X,B)$$  $$\xrightarrow{}  \Ext^{n}_{\textnormal{Ho}(\mathfrak{M})}(X,C)  \xrightarrow{}  \Ext^{n+1}_{\textnormal{Ho}(\mathfrak{M})}(X,A)  \xrightarrow{} \cdots$$
And, similarly there is a contravariant version for the other variable. 

Third, we see in Theorem~\ref{them-Ho-Ext} that the functors $\Ext^n_{\textnormal{Ho}(\mathfrak{M})}(A,B)$ aren't just defined on $\cat{A}$, but descend to functors on $\textnormal{Ho}(\mathfrak{M})$. So they must be something fundamental. Corollary~\ref{cor-Ho-Ext} illuminates this, showing that we have natural isomorphisms:
$$\textnormal{Ho}(\mathfrak{M})(\Omega^n A,B) \cong \Ext^n_{\textnormal{Ho}(\mathfrak{M})}(A,B) \cong \textnormal{Ho}(\mathfrak{M})(A,\Sigma^n B)$$ where $\Omega$  is the loop functor and $\Sigma$ is the suspension functor that we also give easy characterizations of in Appendix~\ref{sec-suspension and loop}.
Thus we have shown that morphism sets in $\textnormal{Ho}(\mathfrak{M})$ may be realized as cohomology groups, computed via canonical (co)resolutions in $\cat{A}$. 

In Section~\ref{sec-Yoneda} we turn to study how $\Ext^n_{\textnormal{Ho}(\mathfrak{M})}$ is related to $\Ext^n_{\cat{A}}$, the usual Yoneda Ext functor. Recall that the latter is defined, regardless of whether or not projective or injective resolutions exist, to be the group of all (equivalence classes of) $n$-fold exact sequences. We show in Theorem~\ref{them-Ho-Yoneda-Ext}/Corollary~\ref{cor-Ext} that there is, for each integer $n \geq 1$, a canonical isomorphism $$\Ext^n_{\textnormal{Ho}(\mathfrak{M})}(A,B) \cong \Ext^n_{\cat{A}}(QA,RB).$$ In particular, for positive integers $n$ we have
$$\textnormal{Ho}(\mathfrak{M})(\Omega^n A,B) \cong \Ext^n_{\cat{A}}(QA,RB) \cong \textnormal{Ho}(\mathfrak{M})(A,\Sigma^n B).$$
We deduce from this, Corollary~\ref{cor-Ho-sets}, which gives natural isomorphisms describing the morphism sets in the homotopy category:
$$\Ext^1_{\cat{A}}(\Sigma QA,RB)   \cong \textnormal{Ho}(\mathfrak{M})(A,B)  \cong \Ext^1_{\cat{A}}(QA,\Omega RB).$$ The interesting thing is that this description of  $\textnormal{Ho}(\mathfrak{M})(A,B)$ is in terms of short exact sequences in $\cat{A}$ and nothing more. Indeed a cofibrant replacement $QA$ is nothing more than any object fitting into a short exact sequence 
$$0 \xrightarrow{} F \xrightarrow{} QA \xrightarrow{} A \xrightarrow{} 0$$  where $QA \in \class{Q}$ and $F \in \class{R}_{\class{W}}$. The dual construction describes how to obtain a fibrant replacement $RB$. Similarly, Appendix~\ref{sec-suspension and loop} shows that any suspension $\Sigma A$ is nothing more than any object fitting into a short exact sequence $$0 \xrightarrow{} A \xrightarrow{} W \xrightarrow{} \Sigma A \xrightarrow{} 0$$ where $W \in \class{W}$. The dual defines $\Omega A$. And, of course, the Yoneda  $\Ext^1_{\cat{A}}$ groups are equivalence classes of short exact sequences, so we have described $\textnormal{Ho}(\mathfrak{M})(A,B) $ completely in terms of short exact sequences in $\cat{A}$. 
 
ln Section~\ref{section-Tor} we consider the case of when $\cat{A}$ comes with a tensor product and develop a similar theory of bifunctors denoted by $\Tor_n^{\textnormal{Ho}(\mathfrak{M})}$.
Finally, Section~\ref{sec-becker} considers how the canonical resolutions defined in this paper relate to some model structures recently constructed by Hanno Becker in~\cite{becker-realization-functor}. There, Becker assumes that $\mathfrak{M} = (\class{Q},\class{W},\class{R})$ is a cofibrantly generated and hereditary abelian model structure on a Grothendieck category $\cat{A}$, and shows that $\mathfrak{M}$ will lift to several Quillen equivalent model structures on $\cha{A}$. Although we don't use the cofibrantly generated and Grothendieck hypothesis for what we do here, we show that our canonical resolutions are exactly cofibrant replacements in one of Becker's model structures on $\cha{A}$, whenever it may exist.

%%%%%%%%%%%%% LEFTOVER  WORDS

%Note that there are four approximation sequences one may attach to any object $A$, given $\mathfrak{M} = (\class{Q},\class{W},\class{R})$. Completeness of the right cotorsion pair $(\class{Q}_{\class{W}},\class{R})$ means we have short exact sequences we construct an  Using completeness of the left cotorsion pair $(\class{Q},\class{R}_{\class{W}})$ 

%The costruction of a canonical resolution $W_{QA}$ uses three of them. Given another object $B$, the fourth approximation can be used to get a fibrant replacement of $RB$. Regardless of all these choices we get a well-defined object... And the process is balanced in the sense that we can use the other objects to construct a canonical coresolution of $B$... 

\subsection{A Remark on Generalities and Exact Categories}\label{sec-exact-cats}
It is stronger than what is needed to insist that $\cat{A}$ be an abelian category. All the results in this paper hold when $\cat{A}$ is just an exact category in the sense of~\cite{quillen-algebraic K-theory}, along with a compatible Hovey triple $\mathfrak{M} = (\class{Q},\class{W},\class{R})$ that is hereditary.
The interested reader with knowledge of exact categories will have no trouble making the minor changes needed to translate the results in this paper to that setting. Indeed the paper has been deliberately written and all proofs presented and proofread to be sure that such translations are immediate. Occasionally we will provide references, usually to~\cite{buhler-exact categories} or~\cite{gillespie-exact model structures}, to support this. For the uninitiated, the best way to learn about exact categories is to read~\cite{buhler-exact categories}.  However, the author has decided to keep the prose directly in terms of abelian categories. First, in most applications of the theory, the ground category tends to be abelian. Second, everybody knows what  an abelian category is, but exact categories, while fundamentally just as easy, are much less known. %To be fair, they \emph{are} a bit more technical.

The following is a simple guide that can be used to translate abelian terminology into the language of exact categories.   

\begin{itemize}
\item Replace the word ``abelian'' with ``exact''. One typically assumes the underlying additive category is also  weakly idempotent complete. See~\cite{buhler-exact categories} and~\cite{gillespie-exact model structures} for terminology. However, this assumption is not even necessary to obtain the results in this paper. Indeed any model structure on an exact category will yield a Hovey triple $\mathfrak{M} = (\class{Q},\class{W},\class{R})$ by~\cite[Theorem~3.3]{gillespie-exact model structures} and this is all that is needed in this paper. The interested reader can find much more on exact model structures in~\cite{stovicek-exact models}.

\item Replace ``monomorphism'' (resp. ``epimorphism'') with ``admissible monomorphism'' or ``inflation'' (resp. ``admissible epimorphism'' or ``deflation''). Interpret ``short exact sequence'' to be a member of the exact structure, that is, an ``admissible short exact sequence'' or a ``conflation''. 

\item The construction of the Yoneda Ext functor $\Ext^n_{\cat{A}}$ will carry through for exact categories $\cat{A}$. While the author is not aware of a source that does this in the language of exact categories, he finds the exposition in~\cite[Chapter VII]{mitchell} to be quite well suited for adapting to exact categories. 

\item Finally, we can define chain complexes, and also exact (or acyclic) chain complexes, in any exact category. This can also be found in~\cite{buhler-exact categories}. 
\end{itemize}

Actually, the Appendix of this paper has been written in the language of exact categories. So the Appendix also serves as an example of how the above translations may be made.

\section{Preliminary notation and terminology}\label{sec-prelims}

Throughout this paper we will consider an abelian category denoted by $\cat{A}$, which the reader may optionally relax to be an exact category as described in Section~\ref{sec-exact-cats}. Typically $\cat{A}$ will possess an abelian model structure $\mathfrak{M} = (\class{Q},\class{W},\class{R})$ that is assumed hereditary. We now give precise definitions of these concepts which will also set the notation used throughout the paper. 

First, given a class of objects $\class{C}$ in $\cat{A}$, the right orthogonal  $\rightperp{\class{C}}$ is defined to be the class of all objects $X$ such that $\textnormal{Ext}^1_{\cat{A}}(C,X) = 0$ for all $C \in \class{C}$. Similarly, we define the left orthogonal $\leftperp{\class{C}}$. A pair of classes $\mathfrak{C} = (\mathcal{X},\mathcal{Y})$ is called a \textbf{cotorsion pair} if $\class{Y} = \rightperp{\class{X}}$ and $\class{X} = \leftperp{\class{Y}}$. We say the cotorsion pair is \textbf{hereditary} if $\textnormal{Ext}^i_{\cat{A}}(X,Y) = 0$ for all $X \in \class{X}$, $Y \in \class{Y}$, and $i \geq 1$. The main practical consequence of the hereditary condition is that the class $\class{X}$ will then be closed under taking kernels of epimorphisms between objects in $\class{X}$; and the dual holds for $\class{Y}$.  In practice, the cotorsion pairs we encounter are typically hereditary. 

We say a cotorsion pair $\mathfrak{C}$ \textbf{has enough projectives} if for each $A \in \cat{A}$ there exists a short exact sequence $0 \rightarrow{} Y \rightarrow{} X \rightarrow{} A \rightarrow{} 0$ with $X \in \class{X}$ and $Y \in \class{Y}$. The epimorphism $X \rightarrow{} A$ is sometimes called a \textbf{special $\class{X}$-precover}. On the other hand, we say $\mathfrak{C}$ \textbf{has enough injectives} if it satisfies the dual, and the monomorphism in that case is sometimes called a \textbf{special $\class{Y}$-preenvelope}. A cotorsion pair $\mathfrak{C}$ is called \textbf{complete} if it has both enough projectives and enough injectives, or in other words, each object has a special $\class{X}$-precover and a special $\class{Y}$-preenvelope. Standard references  for cotorsion pairs include~\cite{enochs-jenda-book} and~\cite{trlifaj-book}.  

The connection between cotorsion pairs on $\cat{A}$ and abelian model structures on $\cat{A}$ can be found in~\cite{hovey}. It is enough to define an \textbf{abelian model structure} on $\cat{A}$ to be a triple $\mathfrak{M} = (\class{Q}, \class{W}, \class{R})$ of classes of objects such that $\class{W}$ is a \textbf{thick} class (meaning it satisfies the \emph{2 out of 3 property} on short exact sequences) such that both  
$(\class{Q}, \class{W}\cap\class{R})$  and $(\class{Q}\cap\class{W}, \class{R})$ are complete cotorsion pairs. We will set $\class{Q}_{\class{W}} := \class{Q} \cap \class{W}$ and  $\class{R}_{\class{W}} := \class{W} \cap \class{R}$. Then $\class{Q}$ is the class of \textbf{cofibrant} objects, $\class{R}$ the class of \textbf{fibrant} objects,  $\class{Q}_{\class{W}}$ the class of \textbf{trivially cofibrant} objects and  $\class{R}_{\class{W}}$ the class of \textbf{trivially fibrant} objects. We also set $\omega := \class{Q}\cap\class{W}\cap\class{R}$ and call this the \textbf{core} of the model structure. A model structure $\mathfrak{M}$ is called \textbf{hereditary} if both of the associated cotorsion pairs $(\class{Q}, \class{R}_{\class{W}})$  and $(\class{Q}_{\class{W}}, \class{R})$ are hereditary. Hereditary abelian  model structures are particularly easy to construct by the result in~\cite{gillespie-hovey triples}. They are also of particular importance, for they are the ones that have appeared in virtually all the applications. See~\cite{gillespie-hereditary-abelian-models} for a recent survey. 

As with any model structure on a suitable category, abelian model structures have homotopy relations. Due to~\cite[Proposition~4.4]{gillespie-exact model structures} they have particularly easy characterizations in the abelian case as follows: Let $f, g : A \xrightarrow{} B$ be two morphisms in $\cat{A}$. We say $f$ and $g$ are \textbf{left homotopic}, written $f \sim^{\ell} g$ if $g-f$ factors through an object in $\class{R}_{\class{W}}$. It is easy to see that $\sim^{\ell}$ is an equivalence relation on $\Hom_{\cat{A}}(A,B)$ called the \textbf{left homotopy relation}.  On the other hand, $f$ and $g$ are \textbf{right homotopic}, written $f \sim^{r} g$, if $g-f$ factors through an object in $\class{Q}_{\class{W}}$. The relation $\sim^{r}$ is called the \textbf{right homotopy relation}.  We say $f$ and $g$ are \textbf{homotopic} or \textbf{fully homotopic}, written $f \sim^{\omega} g$ or simply $f \sim g$, if $g-f$ factors through an object of the core $\omega$.
With all this in mind, we introduce some new terminology that we will use in this paper. First, we will set $\mathfrak{Q} = (\class{Q},\class{R}_{\class{W}})$ and call it the \textbf{left cotorsion pair} and set $\mathfrak{R} = (\class{Q}_{\class{W}},\class{R})$ and call it the \textbf{right cotorsion pair}. This terminology helps us to remember the following: Morphisms $A \xrightarrow{} B$ are \emph{right homotopic} if and only if their difference factors through $\class{Q}_{\class{W}}$, the trivial class appearing in the right cotorsion pair. Moreover, when $B$ (the object on the right) is fibrant, then right homotopy coincides with the full homotopy relation $\sim^{\omega}$. This helps us to distinguish these statements from the dual statements that hold for the left cotorsion pair and left homotopy. The above statements are all due to the characterizations in~\cite[Proposition~4.4]{gillespie-exact model structures} which have proofs that can be readily followed from first principles. 

To each homotopy relation we may associate a stable category. Most importantly, we will write $[f]_{\omega}$, or simply $[f]$ when this is clear, for the homotopy class of $f$. Then $\underline{\Hom}_{\,\omega}(A,B)$ denotes the additive group of all such equivalence classes of morphisms from $A$ to $B$. The reader can verify that there is an additive category, that we denote by $\textnormal{St}_{\omega}(\cat{A})$, whose objects are the same as $\cat{A}$ but whose morphism sets are the groups $\underline{\Hom}_{\,\omega}(A,B)$. We call $\textnormal{St}_{\omega}(\cat{A})$ the \textbf{stable category of $\mathfrak{M}$} and there is a canonical functor we denote by $\gamma_{\omega} : \cat{A} \xrightarrow{} \textnormal{St}_{\omega}(\cat{A})$. We can also discuss the  \textbf{left stable category of $\mathfrak{M}$}, denoted by $\gamma_{\ell} : \cat{A} \xrightarrow{} \textnormal{St}_{\ell}(\cat{A})$, by using the left homotopy classes $[f]_{\ell}$ and the groups $\underline{\Hom}_{\,\ell}(A,B)$. On the other hand, using right homotopy we define $[f]_{r}$ , and $\underline{\Hom}_{\,r}(A,B)$, to obtain $\gamma_{r} : \cat{A} \xrightarrow{} \textnormal{St}_{r}(\cat{A})$, the  \textbf{right stable category of $\mathfrak{M}$} . 

In the setting of an abelian model structure $\mathfrak{M} = (\class{Q},\class{W},\class{R})$, a special $\class{Q}$-precover of an object $A$ is precisely a \textbf{cofibrant replacement} of $A$, and usually denoted by $QA$. That is, a cofibrant replacement of $A$ is constructed by taking any short exact sequence 
$$0 \xrightarrow{} R \xrightarrow{} QA \xrightarrow{} A \xrightarrow{} 0$$ with $QA \in \class{Q}$ and $R \in \class{R}_{\class{W}}$.  We may also refer to such a short exact sequence as a \textbf{cofibrant replacement sequence} for $A$. Similar language applies to fibrant replacements.

Finally, we denote the homotopy category of $\mathfrak{M}$ by $\textnormal{Ho}(\mathfrak{M})$. Following one standard approach, for example see~\cite{dwyer-spalinski} and also comments in~\cite[Section 4]{gillespie-exact model structures}, the morphism sets in $\textnormal{Ho}(\mathfrak{M})$ are defined by setting $$\textnormal{Ho}(\mathfrak{M})(A,B) := \underline{\Hom}_{\,\omega}(RQA,RQB).$$ Different choices of bifibrant replacements, for either $A$ or $B$, are canonically isomorphic, so we are comfortable defining the morphisms sets this way. They are naturally additive abelian groups making $\textnormal{Ho}(\mathfrak{M})$ an additive category with $\class{W}$ the class of zero objects. In Appendix~\ref{sec-suspension and loop} we give a very easy and useful definition of the \emph{loop} functor  $\textnormal{Ho}(\mathfrak{M}) \xrightarrow{\Omega} \textnormal{Ho}(\mathfrak{M})$ and the \emph{suspension} functor $\textnormal{Ho}(\mathfrak{M}) \xrightarrow{\Sigma} \textnormal{Ho}(\mathfrak{M})$, assuming $\mathfrak{M}$ is hereditary.

\section{Homotopy lemmas for augmentations and full resolutions}\label{sec-augmentations-homotopies}

Throughout this section, $\cat{A}$ denotes an abelian category and Ch$(\cat{A})$ the associated category of chain complexes. Again, $\cat{A}$ may even be an exact category as discussed in Section~\ref{sec-exact-cats}.

Let $A \in \cat{A}$. By an \textbf{augmentation of $A$}, denoted $X \xrightarrow{\epsilon} A$, we mean a chain complex $$ \cdots \xrightarrow{} X_2 \xrightarrow{d_2} X_1 \xrightarrow{d_1} X_0 \xrightarrow{\epsilon} A \xrightarrow{} 0.$$ We often will write $X_{\geq0} \xrightarrow{\epsilon} A$ to specify that the domain of $\epsilon$ is in degree 0 and in this way we may also write $X_{\geq i} \xrightarrow{\epsilon} A$ to specify a different degree. If $X \xrightarrow{\epsilon} A$ is an exact (acyclic) complex we call it a \textbf{resolution of $A$}. On the other hand, a \textbf{co-augmentation of $A$}, denoted $A \xrightarrow{\eta} X$ or more specifically $A \xrightarrow{\eta} X_{\leq -1}$, will be a chain complex   $$0 \xrightarrow{} A \xrightarrow{\eta} X_{-1} \xrightarrow{d_{-1}} X_{-2} \xrightarrow{d_{-2}} \cdots ,$$
and we call it a \textbf{coresolution} when it is an exact complex. 
By a \textbf{full augmentation of $A$}, we mean a chain complex $X = X_{\geq 0} \xrightarrow{\eta \epsilon} X_{\leq -1}$ obtained by splicing together (via the composition $d_0 = \eta \epsilon$) an augmentation $X_{\geq 0} \xrightarrow{\epsilon} A$ and a co-augmentation $A \xrightarrow{\eta} X_{\leq -1}$. We call it a \textbf{full resolution of $A$} if $X_{\geq 0} \xrightarrow{\eta\epsilon} X_{\leq-1}$ is an exact complex; this happens if and only if $X_{\geq 0} \xrightarrow{\epsilon} A$ is a resolution and $A \xrightarrow{\eta} X_{\leq-1}$ is a coresolution.

\begin{remark}
The above uses the notation of chain complexes as opposed to \emph{cochain} complexes. We will need to use both, and we slightly modify the language for cochain complexes as explained at the beginning of Section~\ref{sec-cochain-complexes}.
\end{remark}

Given a class of objects $\class{X}$, containing 0, and two morphisms $f, g \in \cat{A}$, we write $f \sim^{\class{X}} g$ to mean $g-f$ factors through some object of $\class{X}$.

\begin{lemma}\label{lemma-augmentation-homotopies}
Let $\class{X}$ be a class of objects in $\cat{A}$, containing 0. Assume $X_{\geq0} \xrightarrow{\epsilon_A} A$ is an augmentation with each $X_n \in \class{X}$ while $Y_{\geq0} \xrightarrow{\epsilon_B} B$ is a resolution with $\ker{\epsilon_B} \in \rightperp{\class{X}}$ and each $\ker{d_n} \in \rightperp{\class{X}}$ for all $n \geq 1$. Then any morphism $f : A \xrightarrow{} B$ extends to a chain map $(X_{\geq0} \xrightarrow{\epsilon_A} A) \xrightarrow{\{f_n\}\cup\{f\}} (Y_{\geq0} \xrightarrow{\epsilon_B} B)$. Moreover, if $f \sim^{\class{X}} g$, then any such extension $\{f_n\}\cup\{f\}$, of $f$, is chain homotopic to any such extension $\{g_n\}\cup\{g\}$, of $g$. In particular, any extension of $f$ is unique up to chain homotopy. 
\end{lemma}

\begin{proof}
The reader can easily verify that the classical argument concerning projective resolutions holds in the exact same way, to construct an extension $\{f_n\}$ of $f$, (and that $\{f_n\}\cup\{f\}$ is unique up to chain homotopy in the special case that $f \sim^{\class{X}} f$).

More generally, assume a morphism $h : A \xrightarrow{} B$ factors through some $W \in \class{X}$, so $h = \beta\alpha$ where $h : A \xrightarrow{\alpha} W \xrightarrow{\beta} B$.  Let $\{h_n\}$ be an extension of $h$. We will now show $\{h_n\}\cup\{h\} \sim^{\class{X}} 0$. First, since $\Ext^1_{\cat{A}}(W,\ker{\epsilon_B}) = 0$, there exists $s : W \xrightarrow{} Y_0$ such that $\epsilon_B s = \beta$.
$$\begin{tikzcd}
\vdots \arrow[d] & &   \vdots \arrow[d] \\
X_1 \arrow[rr,"h_1"] \arrow[d, "d_1" '] &  & Y_1 \arrow[d, "d_1"]  \\
X_0 \arrow[rru,dashed, "s_1"] \arrow[rr,"h_0"] \arrow[d, "\epsilon_A" '] &  & Y_0 \arrow[d, twoheadrightarrow, "\epsilon_B"]  \\
A \arrow[r, "\alpha" ']  \arrow{d} & W \arrow[ru, dashed, "s"]\arrow[r, "\beta" '] & B \arrow{d} \\
0 & & 0
\end{tikzcd}$$
The differential $d_1 : Y_1 \xrightarrow{} Y_0$ factors as $d_1 = k_1e_1$ where $k_1 : \ker{\epsilon_B} \xrightarrow{} Y_0$. Since $\epsilon_B(h_0 - s\alpha\epsilon_A) = 0$, the universal property of the kernel $\ker{\epsilon_B}$ shows that $h_0 - s\alpha\epsilon_A$ factors through the inclusion $k_1 : \ker{\epsilon_B} \xrightarrow{} Y_0$. So now since $\Ext^1_{\cat{A}}(X_0,\ker{d_1}) = 0$, there exists $s_1 : X_0 \xrightarrow{} Y_1$ (a lift over $e_1$) such that $d_1s_1 = h_0 - s\alpha\epsilon_A$. Thus $h_0 = s\alpha\epsilon_A + d_1s_1$. This completes the first step of constructing a null homotopy which one then completes with an induction argument. So, in general, if $f \sim^{\class{X}} g$, we set $h = g-f$ and let $\{h_n\}_{n\geq0} = \{g_n\}_{n\geq0} - \{f_n\}_{n\geq0}$ where $\{f_n\}$ is an extension of $f$ and $\{g_n\}$ of $g$. The above argument provides morphisms $\{s_n : X_{n-1} \xrightarrow{} Y_n\}_{n\geq1}$ satisfying:
\begin{equation}\label{equations-homotopy1}
\begin{aligned} 
h_n &=  d_{n+1}s_{n+1} + s_nd_n \text{ , for all } n \geq 1. \\ 
h_0 &=  d_1s_1 + (s\alpha)\epsilon_A \\
h   &=\epsilon_B(s\alpha) + 0
\end{aligned}
\end{equation}
So here we see $\{h_n\}\cup\{h\}$ is null homotopic by using $s\alpha : A \xrightarrow{} Y_0$ in the homotopy.
\end{proof}

\begin{lemma}[Dual of Lemma~\ref{lemma-augmentation-homotopies}]\label{lemma-dual-augmentation-homotopies}
Let $\class{Y}$ be a class of objects in $\cat{A}$, containing 0. Assume $B \xrightarrow{\eta_B} Y_{\leq{-1}}$ is a co-augmentation with each $Y_n \in \class{Y}$ while $A \xrightarrow{\eta_A} X_{\leq{-1}}$ is a coresolution with $\cok{\eta_A} \in \leftperp{\class{Y}}$ and each $\cok{d_n} \in \leftperp{\class{Y}}$ for all $n \leq -1$. Then any morphism $f : A \xrightarrow{} B$ extends to a chain map $(A \xrightarrow{\eta_A} X_{\leq{-1}}) \xrightarrow{\{f\} \cup \{f_n\}} (B \xrightarrow{\eta_B} Y_{\leq{-1}})$. Moreover, if $f \sim^{\class{Y}} g$, then any such extension $\{f\} \cup \{f_n\}$, of $f$, is chain homotopic to any such extension $\{g\} \cup \{g_n\}$, of $g$. In particular, any extension of $f$ is unique up to chain homotopy. 
\end{lemma}

\begin{proof}
This is the dual of Lemma~\ref{lemma-augmentation-homotopies} and  generalizes the classical result concerning injective coresolutions. We are however using \emph{chain} complex (as opposed to \emph{cochain} complex) notation and will wish to reference this proof in our work ahead. In particular, in the case that $A \xrightarrow{h} B$ factors through some $W \in \class{Y}$, as $h = \beta\alpha$, we use $\Ext^1_{\cat{A}}(\cok{\eta_A}, W) = 0$ to construct a morphism $t : X_{-1} \xrightarrow{} W$ such that $\alpha = t\eta_A$.
$$\begin{tikzcd}
0 \arrow[d] & &  0 \arrow[d] \\
A \arrow[r, "\alpha"]  \arrow[d, rightarrowtail, "\eta_A" '] & W \arrow[r, "\beta"] & B \arrow[d, "\eta_B"] \\
X_{-1} \arrow[ru,dashed, "t"] \arrow[rr,"h_{-1}"] \arrow[d, "d_{-1}" '] &  & Y_{-1} \arrow[d, "d_{-1}"]  \\
X_{-2} \arrow[rr, "h_{-2}" '] \arrow[rru, dashed, "s_{-1}"] \arrow{d} & & Y_{-2} \arrow{d} \\
\vdots & & \vdots
\end{tikzcd}$$
Letting $\{h\}\cup\{h_n\}_{n\leq-1}$ be any extension of $h$ we go on to show it is null homotopic by constructing morphisms $\{s_n : X_{n-1} \xrightarrow{} Y_n\}_{n\leq-1}$ satisfying:
\begin{equation}\label{equations-homotopy2}
\begin{aligned} 
h   &= 0 + (\beta t)\eta_A\\
h_{-1} &=  \eta_B(\beta t) + s_{-1}d_{-1} \\
h_n &=  d_{n+1}s_{n+1} + s_nd_n \text{ , for all } n \leq -2. \\ 
\end{aligned}
\end{equation}
\end{proof}

\begin{lemma}\label{lemma-full-augmentation-homotopies}
Let $\class{X}$ and $\class{Y}$ be  classes of objects in $\cat{A}$, each  containing 0. Let $\omega = \class{X} \cap \class{Y}$. Assume we have objects $A$ and $B$ along with full augmentations $X_A$ and $Y_B$ as follows:
\begin{itemize}
\item The augmentation $X_{\geq0} \xrightarrow{\epsilon_A} A$ has each $X_n \in \class{X}$ while $A \xrightarrow{\eta_A} X_{\leq{-1}}$ is a coresolution with $\cok{\eta_A} \in \leftperp{\class{Y}}$ and each $\cok{d_n} \in \leftperp{\class{Y}}$ for all $n \leq -1$.
\item  The co-augmentation $B \xrightarrow{\eta_B} Y_{\leq{-1}}$ has each $Y_n \in \class{Y}$ while $Y_{\geq0} \xrightarrow{\epsilon_B} B$ is a resolution with $\ker{\epsilon_B} \in \rightperp{\class{X}}$ and each $\ker{d_n} \in \rightperp{\class{X}}$ for all $n \geq 1$.
\end{itemize}
Then any morphism $f : A \xrightarrow{} B$ extends to a chain map $\{f_n\}_{n\in\Z}  : X_A \xrightarrow{} Y_B$ in the sense that $\epsilon_Bf_0 = f\epsilon_A$ and $\eta_Bf = f_{-1}\eta_A$. Moreover, if $f \sim^{\omega} g$, then any such extension $\{f_n\}$ of $f$ is chain homotopic to any such extension $\{g_n\}$ of $g$. In particular, any extension $\{f_n\}$ of $f$ is unique up to chain homotopy. 
\end{lemma}

\begin{proof}
Since $X_{\geq0} \xrightarrow{\epsilon_A} A$ and $Y_{\geq0} \xrightarrow{\epsilon_B} B$ satisfy the hypotheses of Lemma~\ref{lemma-augmentation-homotopies}, while $B \xrightarrow{\eta_B} Y_{\leq{-1}}$ and $A \xrightarrow{\eta_A} X_{\leq{-1}}$ satisfy the hypotheses of Lemma~\ref{lemma-dual-augmentation-homotopies}, those lemmas do provide a full chain map $\{f_n\}_{n\in\Z} : X_A \xrightarrow{} Y_B$ extending $f$ in the sense that $\epsilon_Bf_0 = f\epsilon_A$ and $\eta_Bf = f_{-1}\eta_A$. To complete the proof, it is enough to show that $h \sim^{\omega} 0 \Longrightarrow \{h_n\}_{n\in\Z} \sim 0$ (null homotopic). So suppose $h = \beta\alpha$ where $h : A \xrightarrow{\alpha}W\xrightarrow{\beta} B$ and $W \in \omega$. Then as in the proof of Lemma~\ref{lemma-augmentation-homotopies}, we obtain a morphism $s : W \xrightarrow{} Y_0$ and a collection of morphisms $\{s_n : X_{n-1} \xrightarrow{} Y_n\}_{n\geq1}$ collectively satisfying the equations in~(\ref{equations-homotopy1}):
\begin{equation*}
\begin{aligned} 
h_n &=  d_{n+1}s_{n+1} + s_nd_n \text{ , for all } n \geq 1. \\ 
h_0 &=  d_1s_1 + (s\alpha)\epsilon_A \\
h   &=\epsilon_B(s\alpha) + 0
\end{aligned}
\end{equation*}
On the other hand, the proof of Lemma~\ref{lemma-dual-augmentation-homotopies} provides morphisms $t : X_{-1} \xrightarrow{} W$ and $\{s_n : X_{n-1} \xrightarrow{} Y_n\}_{n\leq-1}$ collectively satisfying the equations in~(\ref{equations-homotopy2}):
\begin{equation*}
\begin{aligned} 
h   &= 0 + (\beta t)\eta_A\\
h_{-1} &=  \eta_B(\beta t) + s_{-1}d_{-1} \\
h_n &=  d_{n+1}s_{n+1} + s_nd_n \text{ , for all } n \leq -2. \\ 
\end{aligned}
\end{equation*}
Pasting the corresponding diagrams in those proofs together we get:
$$
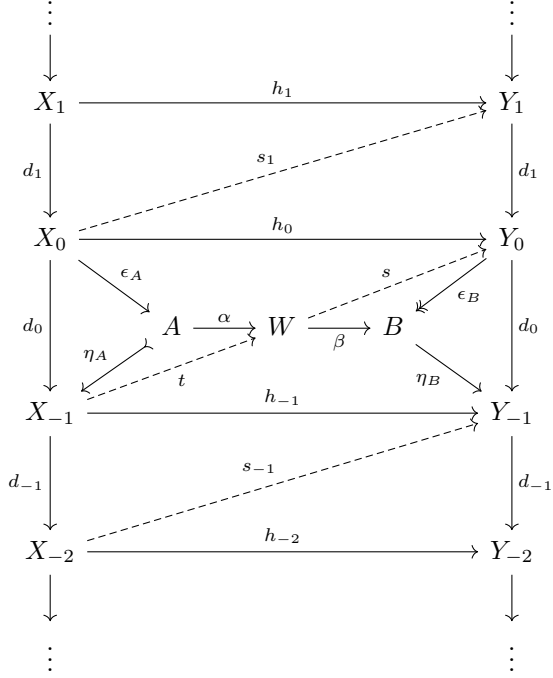
\begin{tikzcd}
\vdots \arrow[d] &&&& \vdots \arrow[d] \\
X_1  \arrow[dd, "d_1" '] \arrow[rrrr, "h_1"]&  &   &  & Y_1 \arrow[dd, "d_1"] \\
&&&&\\
X_0  \arrow[rrrruu, dashed, "s_1"]  \arrow[rrrr, "h_0"] \arrow[rd, "\epsilon_A"] \arrow[dd, "d_0" ']&  &   &  & Y_0 \arrow[ld, two heads, "\epsilon_B"] \arrow[dd, "d_0"]\\
&  A \arrow[r, "\alpha"] \arrow[ld, rightarrowtail, "\eta_A" '] & W \arrow[rru, dashed, "s"] \arrow[r, "\beta" ']& B \arrow[rd, "\eta_B" ']& \\
X_{-1}  \arrow[rru, dashed, "t" '] \arrow[dd, "d_{-1}" '] \arrow[rrrr, "h_{-1}"]&  &   &  & Y_{-1} \arrow[dd, "d_{-1}"]\\
&&&&\\
X_{-2}  \arrow[rrrruu, dashed, "s_{-1}"] \arrow[rrrr, "h_{-2}"] \arrow[d]&  &   &  & Y_{-2} \arrow[d] \\
\vdots &&&& \vdots \\
\end{tikzcd}$$
We now define $s_0$ to be the composition $s_0 : X_{-1} \xrightarrow{t} W \xrightarrow{s} Y_0$ and we claim that 
$$\{s_n : X_{n-1} \xrightarrow{} Y_n\}_{n\geq1} \cup \{s_0 : X_{-1} \xrightarrow{} Y_0\} \cup \{s_n : X_{n-1} \xrightarrow{} Y_n\}_{n\leq-1}$$ provides the desired null homotopy.
For this it is left to show $$h_0 = d_1s_1 + s_0d_0 \ \ \text{  and  } \ \ h_{-1} = d_0s_0 + s_{-1}d_{-1}.$$
But following the diagram we see $$d_1s_1 + s_0d_0 = d_1s_1 + (st)(\eta_A\epsilon_A) = d_1s_1+s\alpha\epsilon_A = h_0.$$
Similarly, $d_0s_0 + s_{-1}d_{-1} = (\eta_B\epsilon_B)(st) + s_{-1}d_{-1} = \eta_B\beta t + s_{-1}d_{-1} = h_{-1}$.
\end{proof}

\subsection{Example concerning the projective and injective stable categories}
Here we give a nice application of the previous lemma which also alludes to some ideas that will be expanded upon in the rest of the paper. 
Let $\cat{A}$ denote any exact category. Let $\class{I}$ denote the class of all injective objects and $\textnormal{St}_{\class{I}}(\cat{A})$ denote the \textbf{injective stable category} of $\cat{A}$. We recall its definition as follows. First, write $f \sim^{\class{I}} g$ to mean that the morphism $g - f$ factors through an injective object. This is an equivalence relation and we let  $\underline{\Hom}_{\,\class{I}}(A,B)$ denote $\Hom_{\cat{A}}(A,B)$ modulo this relation. Given a morphism $f : A \xrightarrow{} B$ we let $[f]_{\class{I}} \in \underline{\Hom}_{\,\class{I}}(A,B)$ denote the equivalence class of $f : A \xrightarrow{} B$.  So now $\textnormal{St}_{\class{I}}(\cat{A})$ is the category whose objects are the same as those in $\cat{A}$ but its morphism sets are the $\underline{\Hom}_{\,\class{I}}(A,B)$ instead of $\Hom_{\cat{A}}(A,B)$. $\textnormal{St}_{\class{I}}(\cat{A})$ is an additive category and we get an additive functor $\gamma_{\class{I}} : \cat{A} \xrightarrow{} \textnormal{St}_{\class{I}}(\cat{A})$ by $f \mapsto [f]_{\class{I}}$. 

Assume $\cat{A}$ has enough injectives. It is well-known that there is an additive functor $\Sigma : \textnormal{St}_{\class{I}}(\cat{A}) \xrightarrow{} \textnormal{St}_{\class{I}}(\cat{A})$, defined on objects by taking  $\Sigma A$ to be an object fitting into a short exact sequence 
$$0\xrightarrow{} A \xrightarrow{} I \xrightarrow{} \Sigma A \xrightarrow{} 0$$ with $I \in \class{I}$. $\Sigma A$ is well-defined up to a canonical isomorphism in $\textnormal{St}_{\class{I}}(\cat{A})$. Similarly morphisms $f :A \xrightarrow{} B$ induce morphisms of such short exact sequences and we get that any morphism $[f]_{\class{I}} : A \xrightarrow{} B$ in $\textnormal{St}_{\class{I}}(\cat{A})$ induces a canonical  morphism $\Sigma([f]_{\class{I}}) : \Sigma A \xrightarrow{} \Sigma B$.  Hence we get a functor $\Sigma$ that is a well-defined up to a canonical isomorphism.  

Now assume $\cat{A}$ has both enough projectives and and enough injectives. 
Let $$\mathbb{X}_A \equiv (P_{\geq 0} \xrightarrow{\epsilon_A} A \xrightarrow{\eta_A} I_{\leq -1})$$ be a full resolution obtained by pasting a projective resolution of $A$ together with an injective coresolution of $A$. Using Lemma~\ref{lemma-full-augmentation-homotopies}, with $\class{X} = \class{P}$ and $\class{Y} = \class{I}$, one argues that  such a full resolution of $A$ is unique up to a canonical chain homotopy equivalence. Similarly, one argues that the following definition provides a well defined functor, contravariant in $A$ and covariant in $B$.

\begin{definition}\label{definition-left}
$\ell\Ext^n_{\cat{A}}(A,B) = H^n[\Hom_{\cat{A}}(\mathbb{X}_A,B)]$ for all $n \in \Z$.
\end{definition} 

\begin{lemma}\label{lemma-stable-Ext}
Assume $\cat{A}$ has enough projectives and injectives. Then for $n \geq 1$ we have 
$$\ell\Ext^n_{\cat{A}}(A,B) =  \Ext^n_{\cat{A}}(A,B)$$ and for $n \geq 0$ we have 
$$\ell\Ext^{-n}_{\cat{A}}(A,B) =   \underline{\Hom}_{\,\class{I}}(\Sigma^{n}A, B).$$  Moreover, for each $A$, the functor $\ell\Ext^n_{\cat{A}}(A,-)$ descends via $\gamma_{\class{I}} : \cat{A} \xrightarrow{} \textnormal{St}_{\class{I}}(\cat{A})$ to a functor $\ell\Ext^n_{\cat{A}}(A,-) : \textnormal{St}_{\class{I}}(\cat{A}) \xrightarrow{} \textbf{Ab}$, where $\textbf{Ab}$ is abelian groups.
\end{lemma}

\begin{proof}
The first statement is just the standard fact that we can compute $\Ext^n_{\cat{A}}$ using projective resolutions. We leave it to the reader to verify the second statement. However, we point out that a more general argument of the same nature appears ahead within the proof of Theorem~\ref{them-lExt and rExt}.

We describe two ways to see that we get a functor  $\ell\Ext^n_{\cat{A}}(A,-) : \textnormal{St}_{\class{I}}(\cat{A}) \xrightarrow{} \textbf{Ab}$. First, one can now proceed to show that the Yoneda Ext functor $\Ext^n_{\cat{A}}(A,-) : \cat{A} \xrightarrow{} \textnormal{Ab}$ descends to an additive functor $\Ext^n_{\cat{A}}(A,-) : \textnormal{St}_{\class{I}}(\cat{A}) \xrightarrow{} \textnormal{Ab}$ by factoring through $\gamma_{\class{I}} : \cat{A} \xrightarrow{} \textnormal{St}_{\class{I}}(\cat{A})$. 
This is a special case of Proposition~\ref{descending-Ext-functor}(1), by taking $\class{M} = (\cat{A},\cat{A},\class{I})$, and the last paragraph of the proof of that proposition indicates the proof of this special case. A second approach is as follows. Suppose a morphism $g : B \xrightarrow{} B'$ factors as $g = (B \xrightarrow{\alpha} J \xrightarrow{\beta} B')$ where $J \in \class{I}$. Then for a full resolution $\mathbb{X}_A$ as in Definition~\ref{definition-left}, the cochain map $\Hom_{\cat{A}}(\mathbb{X}_A , g)$ factors as 
$$\Hom_{\cat{A}}(\mathbb{X}_A , B) \xrightarrow{\Hom_{\cat{A}}(\mathbb{X}_A , \alpha)} \Hom_{\cat{A}}(\mathbb{X}_A , J) \xrightarrow{\Hom_{\cat{A}}(\mathbb{X}_A , \beta)} \Hom_{\cat{A}}(\mathbb{X}_A , B').$$
But $\Hom_{\cat{A}}(\mathbb{X}_A , J)$ is exact since $J$ is injective and it follows that for each $n$ the morphism $\ell\Ext^n_{\cat{A}}(A,g)$ is 0. It follows that the functor $\ell\Ext^n_{\cat{A}}(A, -)$ is compatible with $\gamma_{\class{I}} : \cat{A} \xrightarrow{} \textnormal{St}_{\class{I}}(\cat{A})$.
\end{proof}

\begin{proposition}
Assume $\cat{A}$ has enough projectives and injectives and $Y \in \cat{A}$. Then each short exact sequence $0 \xrightarrow{} A \xrightarrow{} B \xrightarrow{} C \xrightarrow{} 0$ induces a long exact sequence $$ \cdots  \xrightarrow{}   \underline{\Hom}_{\,\class{I}}(\Sigma^2 A,Y)   \xrightarrow{}   \underline{\Hom}_{\,\class{I}}(\Sigma C,Y)  \xrightarrow{}  \underline{\Hom}_{\,\class{I}}(\Sigma B,Y) \xrightarrow{}  \underline{\Hom}_{\,\class{I}}(\Sigma A,Y) \xrightarrow{}  $$  

 $$ \underline{\Hom}_{\,\class{I}}(C,Y) \xrightarrow{} \underline{\Hom}_{\,\class{I}}(B,Y)  \xrightarrow{} \underline{\Hom}_{\,\class{I}}(A,Y) \xrightarrow{}   \Ext^1_{\cat{A}}(C,Y)   \xrightarrow{}  \Ext^1_{\cat{A}}(B,Y) \xrightarrow{} $$

 $$ \Ext^1_{\cat{A}}(A,Y)   \xrightarrow{}  \Ext^2_{\cat{A}}(C,Y)  \xrightarrow{}  \Ext^2_{\cat{A}}(B,Y)    \xrightarrow{} \Ext^2_{\cat{A}}(A,Y)  \xrightarrow{}   \Ext^3_{\cat{A}}(C,Y)  \xrightarrow{} \cdots $$
\end{proposition}

\begin{proof}
Applying the usual horseshoe lemmas (both the projective and injective versions) leads us to a degreewise split short exact sequence
of chain complexes as indicated  below.
$$\begin{CD}
     0   @>>>    P_{\geq 0}   @>>>   P_{\geq 0}    @>>>    P_{\geq 0}    @>>>    0 \\
    @.        @V\eta_A\circ \epsilon_AVV     @V\eta_B\circ \epsilon_BVV      @V\eta_C\circ \epsilon_CVV   @.\\
        0   @>>>    I_{\leq -1}   @>>>   I_{\leq -1}    @>>>   I_{\leq -1}'    @>>>    0 \\
    \end{CD}$$
Since it is degreewise split it remains exact after applying any $\Hom_{\cat{A}}(- , Y)$. Then 
the fundamental lemma of homological algebra leads us to a long exact sequence in $\ell\Ext^n_{\cat{A}}$ which according to Lemma~\ref{lemma-stable-Ext} proves the result.
\end{proof}  

Everything above can be dualized. So $\class{I}$ gets replaced with the class $\class{P}$ of all projective objects to obtain $\textnormal{St}_{\class{P}}(\cat{A})$, the \textbf{projective stable category} of $\cat{A}$. Then instead of the full resolutions $\mathbb{X}_A$ we use cochain complex notation, setting $$\mathbb{X}^A \equiv (\class{P}^{\leq -1} \xrightarrow{\epsilon_A} A \xrightarrow{\eta_A} \class{I}^{\geq 0})$$ to obtain a full coresolution of $A$ obtained by pasting a projective resolution of $A$ together with an injective coresolution of $A$. (See Section~\ref{sec-cochain-complexes} for the language and notation we are using here.) We then replace Definition~\ref{definition-left} with $$r\Ext^n_{\cat{A}}(A,B) = H^n[\Hom_{\cat{A}}(A,\mathbb{X}^B)]$$
and replace $\Sigma$ with the functor $\Omega : \textnormal{St}_{\class{P}}(\cat{A}) \xrightarrow{} \textnormal{St}_{\class{P}}(\cat{A})$ obtained by using enough projectives.
We then get that each short exact sequence $0 \xrightarrow{} A \xrightarrow{} B \xrightarrow{} C \xrightarrow{} 0$, and object $X \in \cat{A}$, induces a long exact sequence $$ \cdots  \xrightarrow{}   \underline{\Hom}_{\,\class{P}}(X, \Omega^2 C)   \xrightarrow{}   \underline{\Hom}_{\,\class{P}}(X, \Omega A)  \xrightarrow{}  \underline{\Hom}_{\,\class{P}}(X, \Omega B) \xrightarrow{}  \underline{\Hom}_{\,\class{P}}(X, \Omega C) \xrightarrow{}  $$  
 $$ \underline{\Hom}_{\,\class{P}}(X, A) \xrightarrow{} \underline{\Hom}_{\,\class{P}}(X,B)  \xrightarrow{} \underline{\Hom}_{\,\class{P}}(X, C) \xrightarrow{}   \Ext^1_{\cat{A}}(X,A)   \xrightarrow{}  \Ext^1_{\cat{A}}(X, B) \xrightarrow{} $$ 
 $$ \Ext^1_{\cat{A}}(X, C)   \xrightarrow{}  \Ext^2_{\cat{A}}(X, A)  \xrightarrow{}  \Ext^2_{\cat{A}}(X, B)    \xrightarrow{} \Ext^2_{\cat{A}}(X, C)  \xrightarrow{}   \Ext^3_{\cat{A}}(X,A)  \xrightarrow{} \cdots $$

\section{The functors $\ell\Ext^n_{\mathfrak{M}}$}\label{sec-canonical-resolutions}

Throughout this section, $\mathfrak{M} = (\class{Q},\class{W},\class{R})$ denotes an hereditary abelian model structure on an abelian category $\cat{A}$, and we let $\omega$ denote its core $\omega :=\class{Q}\cap\class{W}\cap\class{R}$. Again, $\cat{A}$ can even be an exact category and $\mathfrak{M}$ an exact model structure in the sense of~\cite{gillespie-exact model structures}; see Section~\ref{sec-exact-cats}.

In the Introduction we defined what we mean by an $\mathfrak{M}$-resolution of an object $A \in \cat{A}$. In the language of the previous section, an \textbf{$\mathfrak{M}$-resolution}, $W_A$, is a full resolution 
$$W_A \equiv (W_{n\geq0} \xrightarrow{\epsilon_A} A \xrightarrow{\eta_A} W_{n\leq-1})$$
where $W_{n\geq0} \xrightarrow{\epsilon_A} A$ is a resolution constructed by using enough projectives of $(\class{Q}_{\class{W}},\class{R})$, and $A \xrightarrow{\eta_A} W_{n\leq-1}$ is a coresolution constructed by using enough injectives of $(\class{Q},\class{R}_{\class{W}})$.

\begin{theorem}[Comparison Theorem]\label{them-comparison-theorem}
Let $W_A$ and $W_B$ denote any $\mathfrak{M}$-resolutions of some objects $A, B \in \cat{A}$. Then for any morphism $f : A \xrightarrow{} B$, there exists a chain map $\{f_n\}_{n\in\Z} : W_A \xrightarrow{} W_B$ extending $f$ in the sense that $\epsilon_Bf_0 = f\epsilon_A$ and $\eta_Bf = f_{-1}\eta_A$. If $f \sim^{\omega} g$, then any such extension $\{f_n\}$ of $f$ is chain homotopic to any such extension $\{g_n\}$ of $g$. In particular, any extension $\{f_n\}$ of $f$ is unique up to chain homotopy.   
\end{theorem}

\begin{proof}
It follows immediately from Lemma~\ref {lemma-full-augmentation-homotopies} by taking $\class{X} = \class{Q}_{\class{W}}$ (the trivially cofibrant objects) and $\class{Y} = \class{R}_{\class{W}}$ (the trivially fibrant objects).
\end{proof}

\begin{remark}
Note that the above Comparison Theorem doesn't require the full hypotheses that $W_A$ and $W_B$ be $\mathfrak{M}$-resolutions. We only need:

(i) $W_{\geq0}\xrightarrow{\epsilon_A} A$ is an augmentation with $W_n$ trivially cofibrant for all $n\geq0$.

(ii) $A \xrightarrow{\eta_A} W_{\leq-1}$ is a coresolution with $\cok{\eta_A}$ and all $\cok{d_n}$ cofibrant for $n\leq-1$.

(iii) $W_{\geq0} \xrightarrow{\epsilon_B} B$ is a resolution with $\ker{\epsilon_B}$ and all $\ker{d_n}$ fibrant for all $n\geq 1$.

(iv) $B \xrightarrow{\eta_B} W_{\leq-1}$ is a co-augmentation with $W_n$ trivially fibrant for all $n \leq-1$. 
\end{remark}

Recall that $K(\cat{A})$ denotes the chain homotopy category of $\cat{A}$. Its objects are chain complexes and its morphisms are chain maps modulo the chain homotopy relation. A chain map $\alpha$ is called a \emph{chain homotopy equivalence} if $[\alpha]$ is an isomorphism in $K(\cat{A})$.

\begin{corollary}\label{cor-K(A)}
An $\mathfrak{M}$-resolution, $W_A$, is unique up to a canonical isomorphism in $K(\cat{A})$. The association $A \mapsto W_A$ defines a functor $\cat{A} \xrightarrow{}  K(\cat{A})$ which descends to a functor $\textnormal{St}_{\omega}(\cat{A})\xrightarrow{}  K(\cat{A})$ by factoring through $\gamma_{\omega} : \cat{A} \xrightarrow{} \textnormal{St}_{\omega}(\cat{A})$.
\end{corollary}

\begin{proof}
This follows from standard arguments using Theorem~\ref{them-comparison-theorem}. For example, if $W_A$ and $W'_A$ are two different $\mathfrak{M}$-resolutions of $A$, then $1_A$ uniquely extends to a morphism $[\alpha]: W_A \xrightarrow{} W'_A$ in $K(\cat{A})$. On the other hand, it must have a reversal $[\beta]: W'_A \xrightarrow{} W_A$ and $\alpha$ and $\beta$ are inverse chain homotopy equivalences since we must have $\alpha\beta \sim 1_{W'_A}$ and $\beta\alpha \sim 1_{W_A}$. So the choice of object $W_A$ is unique up to a canonical isomorphism in $K(\cat{A})$ and now one can go on to use similar arguments to verify that $A \mapsto W_A$ is a functor. 
\end{proof}

The most important case will be when we take an $\mathfrak{M}$-resolution $W_A$ of a cofibrant object $A$. They are characterized in the following lemma. 

\begin{lemma}\label{lemma-can-resolutions-cofibrant}
The following are equivalent for an object $A \in \cat{A}$.
\begin{enumerate}
\item $A$ is cofibrant (resp. trivially cofibrant).
\item Every cycle of any $\mathfrak{M}$-resolution $W_A$ is cofibrant (resp. trivially cofibrant).
\item $\Hom_{\cat{A}}(W_A,R)$ remains an exact complex for any $\mathfrak{M}$-resolution $W_A$  and any trivially fibrant $R \in \class{R}_{\class{W}}$ (resp. fibrant $R \in \class{R}$). 
\item Every component (resp. every cycle) of any $\mathfrak{M}$-resolution $W_A$ is trivially cofibrant.
\end{enumerate}
\end{lemma}

\begin{proof}
(1) $\implies$ (2). Each cokernel is already cofibrant in the coresolution $A \xrightarrow{\eta_A} W_{\leq-1}$. So if $A$ is cofibrant then \emph{all} cycles of $W_A$ are cofibrant, by the hereditary condition. Moreover, if $A$ is trivially cofibrant then all cycles are also trivial by the thickness condition on $\class{W}$.

(2) $\implies$ (3). Set $W = W_A$. We have short exact sequences
$$0 \xrightarrow{} Z_nW \xrightarrow{} W_n \xrightarrow{} Z_{n-1}W \xrightarrow{} 0.$$
So for any $R \in \class{R}_{\class{W}}$ (resp. $R \in \class{R}$), we have exact sequences
$$0 \xrightarrow{} \Hom_{\cat{A}}(Z_{n-1}W,R) \xrightarrow{} \Hom_{\cat{A}}(W_n,R)  \xrightarrow{} \Hom_{\cat{A}}(Z_{n}W,R)  \xrightarrow{} \Ext^1_{\cat{A}}(Z_{n-1}W,R) .$$ 
But $\Ext^1_{\cat{A}}(Z_{n-1}W,R)  = 0$, and this proves (3).

(3) $\implies$ (1). We have the short exact sequence 
$0 \xrightarrow{} \ker{\epsilon_A} \xrightarrow{} Q_0 \xrightarrow{} A \xrightarrow{} 0$ with
$\ker{\epsilon_A}$ fibrant and  $Q_0$ trivially cofibrant. So for each trivially fibrant $R$ (resp.  fibrant $R$), we have an exact sequence $$\Hom_{\cat{A}}(Q_0,R)  \xrightarrow{} \Hom_{\cat{A}}(\ker{\epsilon_A},R)  \xrightarrow{} \Ext^1_{\cat{A}}(A,R)  \xrightarrow{} \Ext^1_{\cat{A}}(Q_0,R) = 0.$$ 
So if (3) holds then it follows that $\Ext^1_{\cat{A}}(A,R) = 0$. So $A$ is cofibrant (resp. trivially cofibrant).  

It is only left to show that $A$ is cofibrant if and only if all terms of $W = W_A$ are trivially cofibrant. Note that since cofibrant objects are closed under extensions, $W_{-1}$ is the only possible non-cofibrant component of $W$. Considering the short exact sequence 
$0 \xrightarrow{} A \xrightarrow{\eta_A} W_{-1} \xrightarrow{} \cok{\eta_A}  \xrightarrow{} 0 $, which has $\cok{\eta_A}$ cofibrant, the hereditary condition guarantees that $A$ is cofibrant if and only if $W_{-1}$ is (trivially) cofibrant. 
\end{proof}

We recall the following definition  from the Introduction. 

\begin{definition}
A \textbf{canonical resolution} of an object $A$ is an $\mathfrak{M}$-resolution $W_{QA}$ of any cofibrant replacement $QA$ of $A$. 
\end{definition}

Thus Lemma~\ref{lemma-can-resolutions-cofibrant} lists properties of canonical resolutions. We note that if $A$ is already cofibrant we may use the terms \emph{$\mathfrak{M}$-resolution of $A$} and \emph{canonical resolution of $A$} interchangeably.

\begin{lemma}[Horseshoe Lemma]\label{lemma-horseshoe-lemma}
Any short exact sequence $$0 \xrightarrow{} A \xrightarrow{f} B \xrightarrow{g} C \xrightarrow{} 0$$ in $\cat{A}$ extends to a short exact sequence of $\mathfrak{M}$-resolutions $$0 \xrightarrow{} W_A \xrightarrow{\{f_n\}} W_B \xrightarrow{\{g_n\}} W_C \xrightarrow{} 0$$ in $\cha{A}$. 
\end{lemma}

\begin{proof}
Since the cotorsion pairs $(\class{Q}_{\class{W}}, \class{R})$ and $(\class{Q}, \class{R}_{\class{W}})$ are hereditary, it follows easily from the generalized horseshoe lemma~\cite[Lemma~1.4.4]{becker}. This lemma has appeared in a few places in the literature, and a different proof is in~\cite{gillespie-stable-cats-cotorsion-pairs}.
\end{proof}

\begin{remark}
Lemma~\ref{lemma-can-resolutions-cofibrant} requires the hereditary hypothesis on $\mathfrak{M}$ and this is the first place we used this assumption in the paper. It is interesting to note that the above Horseshoe Lemma~\ref{lemma-horseshoe-lemma} can be proved without the hereditary hypothesis in the special case that $A$ is fibrant and $C$ is cofibrant. (Indeed one can imitate the usual Horseshoe Lemma for projective resolutions to build up, and the dual to build down.) Furthermore, if $C$ is cofibrant then $W_C$ will still always have trivially cofibrant components without the hereditary hypothesis. However, this still doesn't seem to be enough to get where we wish to go in this paper! 
\end{remark}

\begin{definition}\label{def-lExt}
For objects $A, B \in \cat{A}$, we define $$\ell\Ext^n_{\mathfrak{M}}(A,B) := H^n[\Hom_{\cat{A}}(W_A,B)],$$ where $W_A$ is an $\mathfrak{M}$-resolution of $A$.  A different choice of  $W_A$ would yield a canonical isomorphism of abelian groups.
\end{definition}

Let us prove the statement that  $\ell\Ext^n_{\mathfrak{M}}(A,B)$ is well-defined (up to a canonical isomorphism), regardless of the choice of resolution $W_A$. Indeed if $W_A$ and $W'_A$ are two $\mathfrak{M}$-resolutions of $A$, then by Corollary~\ref{cor-K(A)} there is a canonical isomorphism $[\alpha] : W_A \xrightarrow{} W'_A$ in $K(\cat{A})$. For any other object $B \in \cat{A}$, since $\Hom_{\cat{A}}(-,B)$ is a contravariant additive functor, it takes $\alpha$ to a chain homotopy equivalence $\Hom_{\cat{A}}(W'_A,B) \xrightarrow{\alpha^*_B} \Hom_{\cat{A}}(W_A,B)$ of \emph{cochain} complexes of abelian groups. It follows that we have a canonical isomorphism: 
$$H^n[\Hom_{\cat{A}}(W'_A,B)] \xrightarrow{H^n([\alpha^*_B])} H^n[\Hom_{\cat{A}}(W_A,B)].$$

\begin{theorem}\label{them-lExt}
$\ell\Ext^n_{\mathfrak{M}}(A,B)$ is a covariant additive functor in variable $B$ and a contravariant additive functor in variable $A$. It satisfies the following properties:
\begin{enumerate}
\item If $A$ is cofibrant, then each fibration $B \xrightarrow{g} B'$ induces a long exact cohomology sequence:  
$$ 
\cdots  \xrightarrow{} \ell\Ext^{n-1}_{\mathfrak{M}}(A,B') \xrightarrow{}   \ell\Ext^n_{\mathfrak{M}}(A,\ker{g})   \xrightarrow{}  \ell\Ext^n_{\mathfrak{M}}(A,B) $$  $$\xrightarrow{} \ell\Ext^n_{\mathfrak{M}}(A,B')   \xrightarrow{} \ell\Ext^{n+1}_{\mathfrak{M}}(A, \ker{g})   \xrightarrow{} \cdots$$
Moreover, each functor $\ell\Ext^n_{\mathfrak{M}}(A,-)$ identifies left homotopic maps whenever $A$ is cofibrant; in particular it factors through $\gamma : \cat{A} \xrightarrow{} \textnormal{St}_{\omega}(\cat{A})$.

\item If $B$ is fibrant, then each cofibration $A \xrightarrow{f} A'$ induces a long exact cohomology sequence:  $$ 
\cdots  \xrightarrow{} \ell\Ext^{n-1}_{\mathfrak{M}}(A,B) \xrightarrow{}   \ell\Ext^n_{\mathfrak{M}}(\cok{f}, B)   \xrightarrow{}  \ell\Ext^n_{\mathfrak{M}}(A',B) $$  $$\xrightarrow{} \ell\Ext^n_{\mathfrak{M}}(A,B)   \xrightarrow{} \ell\Ext^{n+1}_{\mathfrak{M}}(\cok{f}, B)   \xrightarrow{} \cdots$$
Moreover, for any $B$, $\ell\Ext^n_{\mathfrak{M}}(-,B)$ factors through $\gamma : \cat{A} \xrightarrow{} \textnormal{St}_{\omega}(\cat{A})$.
\end{enumerate}
\end{theorem}

\begin{proof}
For any fixed resolution $W_A$ of $A$ we have the functor $H^n[\Hom_{\cat{A}}(W_A,-)]$ which is covariant and additive and so $\ell\Ext^n_{\mathfrak{M}}(A,-) := H^n[\Hom_{\cat{A}}(W_A,-)]$ is such a functor. For a different choice of resolution $W'_A$ then going back to the paragraph after Definition~\ref{def-lExt}, one can verify that the $[\alpha^*_B]$ assemble to provide a natural isomorphism of functors $$\{H^n([\alpha^*_B])\} : H^n[\Hom_{\cat{A}}(W'_A,-)] \xrightarrow{} H^n[\Hom_{\cat{A}}(W_A,-)].$$ Thus $\ell\Ext^n_{\mathfrak{M}}(A,-)$ is a functor that is well-defined up to canonical isomorphism.
 
 We now go on to prove statement~(1). So suppose $A$ is cofibrant and $$0 \xrightarrow{} F \xrightarrow{f} B \xrightarrow{g} B' \xrightarrow{} 0$$ is a short exact sequence with $F$ fibrant. By Lemma~\ref{lemma-can-resolutions-cofibrant}~(4), we know that all components of $W_A$ are trivially cofibrant. So since $F$ is fibrant, applying $\Hom_{\cat{A}}(W_A,-)$ gives us a short exact sequence of cochain complexes of abelian groups
 $$0 \xrightarrow{} \Hom_{\cat{A}}(W_A, F) \xrightarrow{f_*} \Hom_{\cat{A}}(W_A, B) \xrightarrow{g_*} \Hom_{\cat{A}}(W_A,B') \xrightarrow{} 0.$$ Now applying cohomology $H^n$, the fundamental lemma of homological algebra provides the long exact sequence of $\ell\Ext^n_{\mathfrak{M}}$ groups. 

To complete the proof of (1) we now show that $\ell\Ext^n_{\mathfrak{M}}(A,-)$ identifies left homotopic maps. We recall that two maps are left homotopic, written $f \sim^{\ell} g$, if and only if $g-f$ factors through a trivially fibrant object. We wish to show that $\ell\Ext^n_{\mathfrak{M}}(A,f) =  \ell\Ext^n_{\mathfrak{M}}(A,g)$ for such maps. But since  $\ell\Ext^n_{\mathfrak{M}}(A,-)$ is an additive functor it is enough to show  $\ell\Ext^n_{\mathfrak{M}}(A,h) = 0$ whenever $h \sim^{\ell} 0$. So suppose $h = \beta\alpha$ where $B_1 \xrightarrow{\alpha} R \xrightarrow{\beta} B_2$ and $R \in \class{R}_{\class{W}}$.
Then we get $$\Hom_{\cat{A}}(W_A, B_1) \xrightarrow{\alpha_*} \Hom_{\cat{A}}(W_A, R) \xrightarrow{\beta_*} \Hom_{\cat{A}}(W_A,B_2).$$ By Lemma~\ref{lemma-can-resolutions-cofibrant}~(3), the complex $\Hom_{\cat{A}}(W_A, R)$ is acyclic, and so $\ell\Ext^n_{\mathfrak{M}}(A,R) = 0$ for all $n$. This proves that $\ell\Ext^n_{\mathfrak{M}}(A,h) :  \ell\Ext^n_{\mathfrak{M}}(A,B_1)  \xrightarrow{}  \ell\Ext^n_{\mathfrak{M}}(A,B_2)$ is 0 for all $n$.

We now turn to show that $\ell\Ext^n_{\mathfrak{M}}(A,B)$ is a functor in variable $A$ and prove the statements in~(2). First, Corollary~\ref{cor-K(A)} tells us that $\mathfrak{M}$-resolutions provide a functor, $\cat{A} \xrightarrow{} K(\cat{A})$, well-defined on objects up to a canonical isomorphism, and that it factors as $\cat{A} \xrightarrow{\gamma_{\omega}} \textnormal{St}_{\omega}(\cat{A}) \xrightarrow{} K(\cat{A})$.  For a fixed $B$, since $\Hom_{\cat{A}}(-, B)$ is a contravariant additive functor it induces a functor $\Hom_{\cat{A}}(-, B) : K(\cat{A}) \xrightarrow{} K(\textbf{Ab})$, where again $K(\textbf{Ab})$ is the homotopy category of cochain complexes of abelian groups. Finally, composing all these functors with  cohomology gives us a contravariant additive functor
$$\cat{A} \xrightarrow{\gamma_{\omega}} \textnormal{St}_{\omega}(\cat{A}) \xrightarrow{} K(\cat{A})  \xrightarrow{\Hom_{\cat{A}}(-, B)} K(\textbf{Ab})   \xrightarrow{H^n} \textbf{Ab}, $$
which is precisely the functor $\ell\Ext^n_{\mathfrak{M}}(-,B)$.

Now suppose $B$ is fibrant and $$\class{E} : 0 \xrightarrow{} A \xrightarrow{f} A' \xrightarrow{g} C \xrightarrow{} 0$$ is a short exact sequence with $C$ cofibrant. Then using the Horseshoe Lemma~\ref{lemma-horseshoe-lemma} we may construct a short exact sequence naturally extending $\class{E}$ to a short exact sequence of $\mathfrak{M}$-resolutions $$0 \xrightarrow{} W_A \xrightarrow{\{f_n\}} W_{A'} \xrightarrow{\{g_n\}} W_{C} \xrightarrow{} 0.$$ 
Since $C$ is cofibrant, all components of $W_C$ are trivially cofibrant by Lemma~\ref{lemma-can-resolutions-cofibrant}~(4). So since $B$ is fibrant, applying $\Hom_{\cat{A}}(-,B)$ gives us another short exact sequence
 $$0 \xrightarrow{} \Hom_{\cat{A}}(W_C, B) \xrightarrow{\{g^*_n\}} \Hom_{\cat{A}}(W_{A'}, B) \xrightarrow{\{f^*_n\}} \Hom_{\cat{A}}(W_A,B) \xrightarrow{} 0.$$ Finally, applying cohomology $H^n$, the fundamental lemma of homological algebra provides the long exact sequence of $\ell\Ext^n_{\mathfrak{M}}$ groups. 
\end{proof}

We record the following easy lemma concerning the vanishing of $\ell\Ext^n_{\mathfrak{M}}$ groups.

\begin{lemma}\label{lemma-lExt-vanishing}
If $A$ is cofibrant (resp. trivially cofibrant) and $B$ is trivially fibrant (resp. fibrant), then $\ell\Ext^n_{\mathfrak{M}}(A,B) = 0$ for all $n \in \Z$. 
\end{lemma}

\begin{proof}
It follows immediately from Lemma~\ref{lemma-can-resolutions-cofibrant}~(1)+(3). That lemma uses the hereditary hypothesis. But note that for the case $A$ trivially cofibrant, the nonnegative portion of $W_A$ is contractible, and it follows that all cycles of $W_A$ are trivially cofibrant.  Thus $\Hom_{\cat{A}}(W_A,B)$ will remain exact for any fibrant $B$ even without the hereditary hypothesis. 
\end{proof}

\section{The functors $r\Ext^n_{\mathfrak{M}}$}\label{sec-cochain-complexes}

Everything we have done in Sections~\ref{sec-augmentations-homotopies} and~\ref{sec-canonical-resolutions} assumes \emph{chain} complex notation. Other times we will wish to use the language and notation of \emph{cochain} complexes. In this notation, the co-augmentations of Section~\ref{sec-augmentations-homotopies} will be written 
$$0 \xrightarrow{} A \xrightarrow{\eta} X^{0} \xrightarrow{d^0} X^1 \xrightarrow{d^1} \cdots ,$$ 
and denoted briefly by $A \xrightarrow{\eta} X^{\geq 0}$, while augmentations will be denoted $X^{\leq -1} \xrightarrow{\epsilon} A$. A  \emph{full co-augmentation} $X \equiv (X^{\leq-1} \xrightarrow{\eta \epsilon} X^{\geq 0})$ is obtained by splicing these together and setting $d^{-1} = \eta \epsilon$, and we call it a \textbf{full coresolution} if it is an exact complex.
Lemmas~\ref{lemma-augmentation-homotopies}-\ref{lemma-full-augmentation-homotopies} each hold in the exact same way, but with the different notation. 

Moving to the analogs of Section~\ref{sec-canonical-resolutions}, by an \textbf{$\mathfrak{M}$-coresolution}, $W^A$, of an object $A \in \cat{A}$, we mean a full coresolution
$$W^A \equiv (W^{n\leq -1} \xrightarrow{\epsilon_A} A \xrightarrow{\eta_A} W^{n\geq 0})$$
where $A \xrightarrow{\eta_A} W^{n\geq 0}$ is a coresolution constructed by using enough injectives of $(\class{Q},\class{R}_{\class{W}})$, and $W^{n\leq -1} \xrightarrow{\epsilon_A} A$ is a resolution constructed by using enough projectives of $(\class{Q}_{\class{W}},\class{R})$. If $A$ is fibrant we refer to such a $W^A$ as a \emph{canonical coresolution}, and for a general object $A$ a \textbf{canonical coresolution} of $A$ refers to a canonical coresolution of any fibrant replacement. The reader can formulate and verify the duals of Theorem~\ref{them-comparison-theorem} through Lemma~\ref{lemma-lExt-vanishing}. For convenience and referencing we will now state the dual of Definition~\ref{def-lExt} and Theorem~\ref{them-lExt}.

\begin{definition}\label{def-rExt}
For objects $A, B \in \cat{A}$, we define $$r\Ext^n_{\mathfrak{M}}(A,B) := H^n[\Hom_{\cat{A}}(A, W^B)],$$ where $W^B$ is an $\mathfrak{M}$-coresolution of $B$.  A different choice of  $W^B$ would yield a canonical isomorphism of abelian groups.
\end{definition}

\begin{theorem}\label{them-rExt}
$r\Ext^n_{\mathfrak{M}}(A,B)$ is a covariant additive functor in variable $B$ and a contravariant additive functor in variable $A$. It satisfies the following properties:
\begin{enumerate}
\item If $B$ is fibrant,  then each cofibration $A \xrightarrow{f} A'$ induces a long exact cohomology sequence:  $$ \cdots  \xrightarrow{} r\Ext^{n-1}_{\mathfrak{M}}(A,B) \xrightarrow{} r\Ext^n_{\mathfrak{M}}(\cok{f}, B)   \xrightarrow{}  r\Ext^n_{\mathfrak{M}}(A',B) $$  $$\xrightarrow{} r\Ext^n_{\mathfrak{M}}(A,B)   \xrightarrow{} r\Ext^{n+1}_{\mathfrak{M}}(\cok{f}, B)   \xrightarrow{} \cdots$$
Moreover, each functor $r\Ext^n_{\mathfrak{M}}(-,B)$ identifies right homotopic maps whenever $B$ is fibrant; in particular it factors through $\gamma : \cat{A} \xrightarrow{} \textnormal{St}_{\omega}(\cat{A})$.

\item If $A$ is cofibrant,  then each fibration $B \xrightarrow{g} B'$ induces a long exact cohomology sequence:  
$$ 
\cdots  \xrightarrow{} r\Ext^{n-1}_{\mathfrak{M}}(A,B') \xrightarrow{}   r\Ext^n_{\mathfrak{M}}(A,\ker{g})   \xrightarrow{}  r\Ext^n_{\mathfrak{M}}(A,B) $$  $$\xrightarrow{} r\Ext^n_{\mathfrak{M}}(A,B')   \xrightarrow{} r\Ext^{n+1}_{\mathfrak{M}}(A, \ker{g})   \xrightarrow{} \cdots$$
Moreover, for any $A$, $r\Ext^n_{\mathfrak{M}}(A,-)$ factors through $\gamma : \cat{A} \xrightarrow{} \textnormal{St}_{\omega}(\cat{A})$.
\end{enumerate}
\end{theorem}

\section{The functors $\Ext^n_{\textnormal{Ho}(\mathfrak{M})}$}\label{sec-Ho-Ext}

In the previous sections we defined the functors $\ell\Ext^n_{\mathfrak{M}}(A,B)$ using an $\mathfrak{M}$-resolution of $A$ and the functors $r\Ext^n_{\mathfrak{M}}(A,B)$ using an $\mathfrak{M}$-coresolution of $B$. They share nice properties when $A$ is cofibrant and $B$ is fibrant and this is because they are in fact naturally isomorphic in this case. This is reminiscent of how the left and right homotopy relations coincide when the source is cofibrant and the target is fibrant. 
One could try and prove directly the canonical isomorphism $$\ell\Ext^n_{\mathfrak{M}}(A,B) \cong r\Ext^n_{\mathfrak{M}}(A,B),$$ for $A$ cofibrant and $B$ fibrant. We instead wish to give another characterization of these functors, see Theorem~\ref{them-lExt and rExt}, from which the isomorphism automatically  follows.  

One may now wish to review Proposition~\ref{prop-suspension functor} of the Appendix.

\begin{definition}\label{def-full-trivial-resolutions}
Let $A \in \cat{A}$. By a \textbf{full trivial resolution of $A$} we mean a full resolution $W \equiv (W_{n\geq 0} \xrightarrow{\epsilon}A\xrightarrow{\eta}W_{n\leq -1})$ of $A$ with all components $W_n \in \class{W}$.  A different choice of a full trivial resolution, $W'$, yields canonical isomorphisms $Z_{n-1}W \cong Z_{n-1}W'$, in $\textnormal{Ho}(\mathfrak{M})$, and these are precisely the objects $\Omega^{n}A$, where we use the convention that $\Omega^{-n}A := \Sigma^nA$ for $n>0$. 

On the other hand, by a \textbf{full trivial coresolution of $A$} we mean a full coresolution $W = (W^{n\leq -1} \xrightarrow{\epsilon}A\xrightarrow{\eta}W^{n\geq 0})$ of $A$ with all components $W^n \in \class{W}$. The cycles of a full trivial coresolution $W$ may be denoted by $\Sigma^nA$, with the convention that $\Sigma^{-n}A := \Omega^nA$ for $n>0$. 
\end{definition}

\begin{proposition}\label{prop-functors-commute}
Regardless of the short exact sequences used to compute the functors $Q$ (cofibrant replacement), $\Omega$ (loop), $R$ (fibrant replacement), and $\Sigma$ (suspension), we have canonical isomorphisms in  $\textnormal{Ho}(\mathfrak{M})$:
\begin{enumerate}
\item $\Omega^nQA \cong Q\Omega^nA$
\item $\Sigma^nQA \cong Q\Sigma^nA$
\item $\Omega^nRA \cong R\Omega^nA$
\item $\Sigma^nRA \cong R\Sigma^nA$
\end{enumerate}
\end{proposition}

\begin{proof}
It is a basic fact that fibrant and cofibrant replacement computations are unique up to canonical isomorphism in $\textnormal{Ho}(\mathfrak{M})$, and the same is true for $\Omega$ and $\Sigma$ by Proposition~\ref{prop-suspension functor}. Now for any full trivial resolution $W$, of $A$, we have short exact sequences
$0 \xrightarrow{} \Omega^{n+1}A   \xrightarrow{}   W_n   \xrightarrow{}   \Omega^nA   \xrightarrow{} 0$. Using the generalized horseshoe lemma, \cite[Lemma~1.4.4]{becker} (the proof holds in exact categories, or see~\cite{gillespie-stable-cats-cotorsion-pairs} for another proof for exact categories), we may construct a  commutative diagram with all rows and columns short exact sequences. 
$$  \begin{CD}
        0 @>>>   R''    @>>>   R_n    @>>>    R'    @>>>    0 \\
    @.        @VVV     @VVV     @VVV   @.\\
    0   @>>>  Q\Omega^{n+1}A     @>>>   QW_n    @>>>  Q\Omega^nA    @>>>  0  \\
    @.        @VVV     @VVV      @VVV   @.\\
    0   @>>>  \Omega^{n+1}A    @>>>   W_n    @>>>    \Omega^nA    @>>>  0  \\
    \end{CD}
$$ and with $R_n,R',R''$ all trivially fibrant. So the middle row represents a cofibrant approximation of the bottom row, and $QW_n \in \class{Q}_{\class{W}}$. By splicing all these short exact sequences together we obtain a short exact sequence of \emph{chain complexes}
$$0 \xrightarrow{} R \xrightarrow{} QW \xrightarrow{} W \xrightarrow{} 0$$
where $R \in \tilclass{R_{\class{W}}}$, the class of all exact complexes with trivially fibrant cycles, and $QW \in \class{C}$, the class of all exact complexes with cofibrant cycles but with trivially cofibrant components. The point here is that the complex $QW$ is a full trivial resolution of $QA$ and that $Z_{n-1}(QW) = Q\Omega^nA$. Referring to Definition~\ref{def-full-trivial-resolutions} we conclude (1) and (2). The proof of (3) and (4) is similar but using fibrant replacements and the notion of a full trivial coresolution. 
\end{proof}

\begin{theorem}\label{them-lExt and rExt}
For $A$ cofibrant and $B$ fibrant, there are natural isomorphisms:
$$\ell\Ext^n_{\mathfrak{M}}(A,B) \cong \textnormal{Ho}(\mathfrak{M})(\Omega^n A,B) \cong \textnormal{Ho}(\mathfrak{M})(A,\Sigma^n B) \cong r\Ext^n_{\mathfrak{M}}(A,B)$$
\end{theorem}

\begin{proof}
We already know, from Proposition~\ref{prop-suspension functor}, that we have natural isomorphisms $\textnormal{Ho}(\mathfrak{M})(\Omega^n A,B) \cong \textnormal{Ho}(\mathfrak{M})(A,\Sigma^n B)$.
 
We recall that the functor $\ell\Ext^n_{\mathfrak{M}}(A,B) := H^n[\Hom_{\cat{A}}(W_A,B)]$ is well-defined up to a canonical isomorphism, where $W_A$ is some canonical resolution of $A$. Because $A$ is cofibrant, we are assured  by Lemma~\ref{lemma-can-resolutions-cofibrant} that $W_A$ has cofibrant cycles, denoted $\Omega^nA$, and trivially cofibrant components, denoted $W_n$. Applying $\Hom_{\cat{A}}(-,B)$ to $W_A$ gives us a cochain complex of abelian groups:
$$\cdots \xleftarrow{} \Hom_{\cat{A}}(W_2,B) \xleftarrow{d^*_2}  \Hom_{\cat{A}}(W_1,B) \xleftarrow{d^*_1} \Hom_{\cat{A}}(W_0,B) \xleftarrow{d^*_0}  \Hom_{\cat{A}}(W_{-1},B) \cdots$$
To compute $\ell\Ext^n_{\mathfrak{M}}(A,B)$ we need $\ker{d^*_{n+1}}/\im{d^*_{n}}$, the cohomology in degree $n$. Now each $W_n \xrightarrow{d_n} W_{n-1}$ factors as $W_n \xrightarrow{\epsilon_n} \Omega^nA \xrightarrow{\eta_{n-1}} W_{n-1}$ where $\epsilon_n $ is an epimorphism and $\eta_{n-1}$ is a monomorphism. We have $\cok{d_{n+1}} = \cok{(\eta_n \epsilon_{n+1})} = \cok{\eta_n} = \epsilon_n$. This implies that we have a left exact sequence of abelian groups
$$0 \xrightarrow{} \Hom_{\cat{A}}(\Omega^nA,B) \xrightarrow{\epsilon^*_n} \Hom_{\cat{A}}(W_n,B)  \xrightarrow{d^*_{n+1}} \Hom_{\cat{A}}(W_{n+1},B).$$ 
In particular, $\epsilon^*_n$ is a monomorphism identifying  $\Hom_{\cat{A}}(\Omega^nA,B)$ with $\ker{d^*_{n+1}}$. Since $B$ is fibrant and $\Omega^nA$ is cofibrant,  we have a natural isomorphism $$\textnormal{Ho}(\mathfrak{M})(\Omega^n A,B) \cong \underline{\Hom}_{\,\omega}(\Omega^n A,B),$$ and so we will have proven $\ell\Ext^n_{\mathfrak{M}}(A,B) \cong \textnormal{Ho}(\mathfrak{M})(\Omega^n A,B)$ once we establish the following. 

\

\noindent \underline{Claim}: The isomorphism $\Hom_{\cat{A}}(\Omega^nA,B) \xrightarrow{\epsilon^*_n}  \ker{d^*_{n+1}} \subseteq \Hom_{\cat{A}}(W_n,B)$ carries the subgroup $\{\,f \in  \Hom_{\cat{A}}(\Omega^nA,B)  \, | \, f \sim^{\omega} 0 \,\}$ onto the subgroup $\im{d^*_n}$. 

\

So suppose $f \in \Hom_{\cat{A}}(\Omega^nA,B)$ satisfies $f \sim^{\omega} 0$. Then $f : \Omega^nA \xrightarrow{} B$ factors as $\Omega^n A \xrightarrow{\alpha} W \xrightarrow{\beta} B$ where $W \in \omega$. Then applying $\Hom_{\class{A}}(-,W)$ to the exact sequence
$$0 \xrightarrow{} \Omega^{n}A \xrightarrow{\eta_{n-1}} W_{n-1} \xrightarrow{\epsilon_{n-1}} \Omega^{n-1}A \xrightarrow{} 0,$$ and using $\Ext^1_{\cat{A}}(\Omega^{n-1}A,W) = 0$, we see that $\alpha$ extends to $W_{n-1}$ through the morphism $\Omega^{n} A  \xrightarrow{\eta_{n-1}} W_{n-1}$. That is, $\alpha = v\eta_{n-1}$ for some $v : W_{n-1} \xrightarrow{} W$. So now $\beta v \in \Hom_{\cat{A}}(W_{n-1}, B)$ and we check $d^*_n(\beta v) = \epsilon^*_n(f)$, proving that $\epsilon^*_n$ maps $f$ into $\im{d^*_n}$.
On the other hand, suppose $f \in \Hom_{\cat{A}}(\Omega^nA,B)$ satisfies $\epsilon_n^*(f) \in \im{d^*_n}$, and we will show $f \sim^{\omega} 0$. Now $\epsilon_n^*(f) \in \im{d^*_n}$ means $f\epsilon_n = td_n$ for some morphism $t : W_{n-1} \xrightarrow{} B$. So then $f\epsilon_n = t\eta_{n-1}\epsilon_n$ and by right canceling $\epsilon_n$ we get  $f= t\eta_{n-1}$. Therefore $f$ factors through $W_{n-1}$, a trivially cofibrant object. But since $B$ is fibrant this is  enough to conclude $f \sim^{\omega} 0$, by~\cite[Proposition~4.4]{gillespie-exact model structures}. 

This proves the Claim, and establishes $\ell\Ext^n_{\mathfrak{M}}(A,B) \cong \textnormal{Ho}(\mathfrak{M})(\Omega^n A,B)$. A similar argument with $r\Ext^n_{\mathfrak{M}}(A,B)$ will establish $\textnormal{Ho}(\mathfrak{M})(A,\Sigma^n B) \cong r\Ext^n_{\mathfrak{M}}(A,B)$.
\end{proof}

Because of the canonical isomorphism $\ell\Ext^n_{\mathfrak{M}}(A,B) \cong r\Ext^n_{\mathfrak{M}}(A,B)$ in Theorem~\ref{them-lExt and rExt} we may, in the case that $A$ is cofibrant and $B$ is fibrant, simply denote this group by $\Ext^n_{\mathfrak{M}}(A,B)$, with the realization that it may be computed by either a canonical resolution of $A$, or, a canonical coresolution of $B$. With this observation we make the following definition. 

\begin{definition}\label{def-Ho-Ext}
$\Ext^n_{\textnormal{Ho}(\mathfrak{M})}(A,B) := \Ext^n_{\mathfrak{M}}(RQA,RQB)$.
\end{definition}

Recall that there is a canonical functor $\gamma_{\mathfrak{M}} : \cat{A} \xrightarrow{} \textnormal{Ho}(\mathfrak{M})$. It is the identity on objects, and takes a morphism $f : A \xrightarrow{} B$ to $[R(Q(f))]_{\omega} \in \underline{\Hom}_{\,\omega}(RQA,RQB)$, where $Q(f)$ and $R(Q(f))$ are any morphisms making the diagram below commute.

$$\begin{CD}
     A    @<p_A<<  QA    @>j_{QA}>>  RQA   \\
     @VVf V     @VV Q(f) V      @VVR(Q(f)) V   \\
    B    @<p_{B}<<   QB    @>j_{QB}>>  RQB     \\
\end{CD}$$

\begin{theorem}\label{them-Ho-Ext}
Each $\Ext^n_{\textnormal{Ho}(\mathfrak{M})}(A,B)$ defines a functor with the following properties.  
\begin{enumerate}
\item $\Ext^n_{\textnormal{Ho}(\mathfrak{M})}(A,-) : \cat{A} \xrightarrow{} \textbf{Ab}$ is a covariant additive functor and short exact sequences in $\cat{A}$ are sent to long exact cohomology sequences of $\Ext^n_{\textnormal{Ho}(\mathfrak{M})}$ groups. $\Ext^n_{\textnormal{Ho}(\mathfrak{M})}(A,-)$  descends via $\gamma_{\mathfrak{M}} : \cat{A} \xrightarrow{} \textnormal{Ho}(\mathfrak{M})$ to a well-defined functor on  $\textnormal{Ho}(\mathfrak{M})$.
\item $\Ext^n_{\textnormal{Ho}(\mathfrak{M})}(-,B) : \cat{A} \xrightarrow{} \textbf{Ab}$ is a contravariant additive functor and short exact sequences in $\cat{A}$ are sent to long exact cohomology sequences of $\Ext^n_{\textnormal{Ho}(\mathfrak{M})}$ groups. $\Ext^n_{\textnormal{Ho}(\mathfrak{M})}(-,B)$ decends via $\gamma_{\mathfrak{M}} : \cat{A} \xrightarrow{} \textnormal{Ho}(\mathfrak{M})$ to a well-defined functor on  $\textnormal{Ho}(\mathfrak{M})$.
\item We have the following isomorphisms, natural in both $A$ and $B$:
 $$\Ext^n_{\textnormal{Ho}(\mathfrak{M})}(A,B) \cong \Ext^n_{\mathfrak{M}}(QA,RB).$$ In particular, the group $\Ext^n_{\textnormal{Ho}(\mathfrak{M})}(A,B)$ may be computed by either applying $H^n[\Hom_{\cat{A}}(-,RB)]$ to a canonical resolution of $A$, or, by applying $H^n[\Hom_{\cat{A}}(QA,-)]$ to a canonical coresolution of $B$.
\end{enumerate}
\end{theorem}

\begin{proof}
(1) Bifibrant replacement determines a covariant additive functor $RQ : \class{A} \xrightarrow{}  \textnormal{St}_{\omega}(\class{A})$. (It is well-defined, with respect to the choice of bifibrant replacement objects used, up to a canonical isomorphism. It even descends to $\textnormal{St}_{\omega}(\class{A})$ via the factorization $\cat{A} \xrightarrow{\gamma_{\omega}} \textnormal{St}_{\omega}(\class{A}) \xrightarrow{RQ}  \textnormal{St}_{\omega}(\class{A})$; this follows from~\cite[Lemmas~2.4 and~2.5]{gillespie-stable-cats-cotorsion-pairs}.) Thus for any object $A$, we obtain a covariant additive functor $r\Ext^n_{\mathfrak{M}}(A,RQ(-)) : \cat{A} \xrightarrow{} \textbf{Ab}$, by Theorem~\ref{them-rExt}~(2). Again, with respect to the choices involved, it is well-defined up to a canonical isomorphism.
In particular, for any fixed bifibrant replacement, $RQA$, we have a covariant additive functor $\Ext^n_{\textnormal{Ho}(\mathfrak{M})}(A,-) := r\Ext^n_{\mathfrak{M}}(RQA,RQ(-))$. 
But if $R'Q'A$ is a different choice of bifibrant replacement of $A$, then we have a canonical stable $\omega$-equivalence $RQA \xrightarrow{\alpha} R'Q'A$ by~\cite[Lemma~4.24]{dwyer-spalinski}. This means $[\alpha]_{\omega}$ is an isomorphism in $\textnormal{St}_{\omega}(\cat{A})$. It then follows from Theorem~\ref{them-rExt}~(1) that we have isomorphisms
$$r\Ext^n_{\mathfrak{M}}(R'Q'A,RQB) \xrightarrow{[\alpha]^*_B} r\Ext^n_{\mathfrak{M}}(RQA,RQB)$$ which assemble to a natural isomorphism 
$$r\Ext^n_{\mathfrak{M}}(R'Q'A,RQ(-)) \xrightarrow{\{[\alpha]^*_B\}} r\Ext^n_{\mathfrak{M}}(RQA,RQ(-)).$$
Thus we have proved that $\Ext^n_{\textnormal{Ho}(\mathfrak{M})}(A,-) : \cat{A} \xrightarrow{} \textbf{Ab}$ is a covariant additive functor and it follows from Theorem~\ref{them-rExt}~(2) along with the generalized horseshoe lemma, \cite[Lemma~1.4.4]{becker} or see~\cite{gillespie-stable-cats-cotorsion-pairs}, that it takes short exact sequences in $\cat{A}$ to long exact sequences in $\textbf{Ab}$.

Now let us define $\Ext^n_{\textnormal{Ho}(\mathfrak{M})}(A,-)$ on $\textnormal{Ho}(\mathfrak{M})$. So say we have a morphism $[f]_{\omega} \in \textnormal{Ho}(\mathfrak{M})(A,B) :=  \underline{\Hom}_{\,\omega}(RQA,RQB)$. Then the definition $\Ext^n_{\textnormal{Ho}(\mathfrak{M})}(A,[f]_{\omega}) := r\Ext^n_{\mathfrak{M}}(RQA,f)$ is well-defined by Theorem~\ref{them-rExt}~(2), and this definition is compatible with $\gamma_{\mathfrak{M}} : \cat{A} \xrightarrow{} \textnormal{Ho}(\mathfrak{M})$. (One could also check that $\Ext^n_{\textnormal{Ho}(\mathfrak{M})}(A,-)$ takes weak equivalences in $\cat{A}$ to isomorphisms of abelian groups, since weak equivalences between bifibrant objects are stable $\omega$-equivalences. So the universal property of $\textnormal{Ho}(\mathfrak{M})$ applies and says that $\Ext^n_{\textnormal{Ho}(\mathfrak{M})}(A,-)$ uniquely factors through $\gamma_{\mathfrak{M}}$. The above makes this factorization explicit.) 

The proof of (2) is similar; one shows that $$\Ext^n_{\textnormal{Ho}(\mathfrak{M})}(-,B) := \ell\Ext^n_{\mathfrak{M}}(RQ(-),RQB)$$ is a well-defined, additive and contravariant functor. 

We now prove (3). First, for any $A$, consider any fibrant approximation sequence $0 \xrightarrow{} QA \xrightarrow{j} RQA \xrightarrow{q} C \xrightarrow{} 0$, of $QA$. Applying $\Ext^n_{\mathfrak{M}}(-,RB)$ and using the associated long exact sequence, along with Lemma~\ref{lemma-lExt-vanishing}, we conclude there is an isomorphism  $j^*_A : \Ext^n_{\mathfrak{M}}(RQA,RB) \xrightarrow{} \Ext^n_{\mathfrak{M}}(QA,RB)$. Now for any morphism $f : A \xrightarrow{} A'$ we get a commutative diagram with exact rows:
$$\begin{CD}
     0   @>>> QA    @>j_A>>  RQA    @>>>  C    @>>>    0 \\
    @.        @VVQ(f) V     @VV R(Q(f)) V      @VVV   @.\\
    0   @>>>  QA'    @>j_{A'}>>   RQA'    @>>>  C'    @>>> 0   \\
\end{CD}$$ Since $\Ext^n_{\mathfrak{M}}(-,RB)$ is a functor, applying it to the left square shows that we have a natural isomorphism $\{j^*_A\} : \Ext^n_{\mathfrak{M}}(RQA,RB) \xrightarrow{} \Ext^n_{\mathfrak{M}}(QA,RB)$. Next, for objects $B$, consider cofibrant approximation sequences $0 \xrightarrow{} F \xrightarrow{i} QRB \xrightarrow{p} RB \xrightarrow{} 0$, of $RB$. A similar argument as above shows that we have a natural isomorphism $\{(p_B)_*\} : \Ext^n_{\mathfrak{M}}(RQA,QRB) \xrightarrow{} \Ext^n_{\mathfrak{M}}(RQA,RB)$. But finally, it is easy to construct a natural transformation $\{\tau_B\} : RQB \xrightarrow{} QRB$, with each $\tau_B$ a weak equivalence. It follows from~\cite[Lemma~4.24]{dwyer-spalinski} that each $\tau_B$ is a stable $\omega$-equivalence. Therefore, by Theorem~\ref{them-rExt}~(2) we have natural isomorphisms $\{(\tau_B)_*\} : \Ext^n_{\mathfrak{M}}(RQA,RQB) \xrightarrow{} \Ext^n_{\mathfrak{M}}(RQA,QRB)$. So composing all these isomorphisms we get a natural isomorphism $$\{j^*_A\} \circ \{(p_B)_*\} \circ \{(\tau_B)_*\} : \Ext^n_{\mathfrak{M}}(RQA,RQB) \xrightarrow{} \Ext^n_{\mathfrak{M}}(QA,RB).$$
\end{proof}

The next corollary, combined with Theorem~\ref{them-Ho-Ext}(3), makes precise how we have proved that morphism sets in $\textnormal{Ho}(\mathfrak{M})$ may be realized as cohomology groups computed via canonical (co)resolutions in $\cat{A}$.

\begin{corollary}\label{cor-Ho-Ext}
 There are natural isomorphisms:
$$\textnormal{Ho}(\mathfrak{M})(\Omega^n A,B) \cong \Ext^n_{\textnormal{Ho}(\mathfrak{M})}(A,B) \cong \textnormal{Ho}(\mathfrak{M})(A,\Sigma^n B)$$
\end{corollary}

\begin{proof}
We have $\Ext^n_{\textnormal{Ho}(\mathfrak{M})}(A,B) \cong \Ext^n_{\mathfrak{M}}(QA,RB) \cong \textnormal{Ho}(\mathfrak{M})(\Omega^n QA,RB)$ by Theorem~\ref{them-Ho-Ext}~(3) and Theorem~\ref{them-lExt and rExt}. Then by Proposition~\ref{prop-functors-commute}~(1), we continue to get $\textnormal{Ho}(\mathfrak{M})(\Omega^n QA,RB) \cong \textnormal{Ho}(\mathfrak{M})(Q\Omega^nA,RB)$. All these isomorphisms are natural and since cofibrant and fibrant replacements also provide natural isomorphisms in $\textnormal{Ho}(\mathfrak{M})$, we conclude that all these are naturally isomorphic to $\textnormal{Ho}(\mathfrak{M})(\Omega^n A,B)$. Similar reasoning shows $\Ext^n_{\textnormal{Ho}(\mathfrak{M})}(A,B) \cong \textnormal{Ho}(\mathfrak{M})(A,\Sigma^n B)$.
\end{proof}

Combining Theorem~\ref{them-Ho-Ext} and Corollary~\ref{cor-Ho-Ext}, we also recover the existence of long exact sequences attached to short exact sequences.

\begin{corollary}\label{cor-long-triangles}
Any short exact sequence $0 \xrightarrow{} A \xrightarrow{} B \xrightarrow{} C \xrightarrow{} 0$ in $\cat{A}$ induces long exact sequences of abelian groups as follows:
\begin{enumerate}
\item 
For each object $X \in \cat{A}$, and setting $\Sigma^{-n} := \Omega^n$ for $n > 0$, there is a long exact sequence:
$$ 
\cdots  \xrightarrow{} \textnormal{Ho}(\mathfrak{M})(X,\Sigma^{n-1} C)  \xrightarrow{}  \textnormal{Ho}(\mathfrak{M})(X,\Sigma^n A)  \xrightarrow{}  \textnormal{Ho}(\mathfrak{M})(X,\Sigma^n B)$$  $$\xrightarrow{}  \textnormal{Ho}(\mathfrak{M})(X,\Sigma^n C)  \xrightarrow{}  \textnormal{Ho}(\mathfrak{M})(X,\Sigma^{n+1} C)  \xrightarrow{} \cdots$$
\item 
For each object $X \in \cat{A}$, and setting $\Omega^{-n} := \Sigma^n$ for $n > 0$, there is a long exact sequence:
$$ 
\cdots  \xrightarrow{} \textnormal{Ho}(\mathfrak{M})(\Omega^{n-1} A, X)  \xrightarrow{}  \textnormal{Ho}(\mathfrak{M})(\Omega^n C,X)  \xrightarrow{}  \textnormal{Ho}(\mathfrak{M})(\Omega^n B,X)$$  $$\xrightarrow{}  \textnormal{Ho}(\mathfrak{M})(\Omega^n A,X)  \xrightarrow{}  \textnormal{Ho}(\mathfrak{M})(\Omega^{n+1} C,X)  \xrightarrow{} \cdots$$
\end{enumerate}
\end{corollary}

For any object $A$ we can find a short exact sequence $0 \xrightarrow{} A \xrightarrow{} W \xrightarrow{} \Sigma A \xrightarrow{} 0$ where $W$ is trivial. Then  since
$\Ext^n_{\textnormal{Ho}(\mathfrak{M})}(-,B)$ vanishes on trivial objects, and sends short exact sequences to long exact sequences, we deduce the following dimension shifting formulas. 

\begin{corollary}\label{cor-dim-shift}
The following dimension shifting formulas hold for all integers $m,n$, 
where we are setting $\Sigma^{-n} := \Omega^n$ and  $\Omega^{-n} := \Sigma^n$ for $n > 0$:
 \[ \Ext^n_{\textnormal{Ho}(\mathfrak{M})}(A,B)  = \begin{cases}
            \Ext^{n+m}_{\textnormal{Ho}(\mathfrak{M})}(\Sigma^mA,B) \\
            
           \\
            
             \Ext^{n+m}_{\textnormal{Ho}(\mathfrak{M})}(A,\Omega^mB).
             \end{cases} \]
\end{corollary}

\begin{example}\label{example-GorenProj}
Assume $\cat{A}$ has enough projectives, and suppose that $(\class{C}, \class{W})$ is a projective cotorsion pair. It means $\mathfrak{M} = (\class{C}, \class{W}, \class{A})$ is an hereditary abelian model structure. There are a variety of such model structures on $R$-Mod, the category of (say left) $R$-modules over a ring $R$, and $\ch$, the category of chain complexes of such $R$-modules, described in~\cite{gillespie-hereditary-abelian-models}. The quintessential example is the cotorsion pair $(\class{G}\class{P}, \rightperp{\class{G}\class{P}})$, where $\class{G}\class{P}$ is the class of Gorenstein projective $R$-modules over an Iwanaga-Gorenstein ring $R$. In any case, any cofibrant $C \in \class{C}$ is a Gorenstein projective object by~\cite[Theorem~5.2/5.4]{gillespie-recollement} and the proof given there reveals that any $\class{M}$-resolution of $C$ provides an exact chain complex, $\mathbb{P}_C$, with $C = Z_{-1}(\mathbb{P}_C)$ and such that $\Hom_{\cat{A}}(\mathbb{P}_C, P)$ remains exact for any projective object $P \in \cat{A}$. Such a complex $\mathbb{P}_C$ is usually called a \emph{totally acyclic complex of projectives}, or, a \emph{complete projective resolution of $C$}. Thus a canonical resolution of a general object $A \in \cat{A}$ provides a complete projective resolution, $\mathbb{P}_{QA}$, of a cofibrant approximation $QA$ of $A$.
Moreover we have natural isomorphisms
$$\textnormal{Ho}(\mathfrak{M})(\Omega^n A,B) \cong H^n[\Hom_{\cat{A}}(\mathbb{P}_{QA}, B)] \cong \textnormal{Ho}(\mathfrak{M})(A,\Sigma^n B)$$
due to Corollary~\ref{cor-Ho-Ext} along with Theorem~\ref{them-Ho-Ext}.

For a general ring $R$, if it turns out that $(\class{G}\class{P}, \rightperp{\class{G}\class{P}})$ is indeed a complete cotorsion pair, then it is automatic by~\cite[Theorem~5.2/5.4]{gillespie-recollement} that we have an hereditary abelian model structure $\mathfrak{M}_{prj}  =(\class{G}\class{P}, \rightperp{\class{G}\class{P}}, R\text{-Mod})$ on $R$-Mod. In this case it is easy to see that a canonical resolution in $\mathfrak{M}_{prj} $ of a Gorenstein projective module is \emph{equivalent} to a complete projective resolution since any cycle of a totally acyclic complex of projectives is, by definition, Gorenstein projective.   
\end{example}

%%%%%%%%%%%%%%%%%%%%%%%%%%%%%%%%%%%%%%%%%%%%%%%%%%%%%%%%%

\section{Relation to the Yoneda Ext functor}\label{sec-Yoneda}

We continue to let $\mathfrak{M} = (\class{Q},\class{W},\class{R})$ denote an hereditary abelian model structure on an abelian category $\cat{A}$ and let $\omega$ denote its core $\omega :=\class{Q}\cap\class{W}\cap\class{R}$. (Again, $\cat{A}$ can in fact be an exact category and $\mathfrak{M}$ an exact model structure; see Section~\ref{sec-exact-cats}). In the previous sections we defined bifunctors $\Ext^n_{\textnormal{Ho}(\mathfrak{M})}$ for all integers $n$. But the usual Yoneda Ext functor, denoted $\Ext^n_{\cat{A}}$, is also defined for all natural numbers $n$. Recall that it is defined, regardless of whether or not projective or injective resolutions exist, to be the group of all (equivalence classes of) $n$-fold exact sequences. See~\cite[Chapter 3]{homology} or~\cite[Chapter VII]{mitchell}. Our aim is to show that $$\Ext^n_{\textnormal{Ho}(\mathfrak{M})}(A,B) \cong \Ext^n_{\cat{A}}(QA,RB),$$ for $n \geq 1$. We first consider some instances that will guarantee that the Yoneda Ext functor descends to a functor on $\textnormal{St}_{\omega}(\cat{A})$, similar to what we saw in Theorems~\ref{them-lExt} and~\ref{them-rExt}.

\begin{proposition}\label{descending-Ext-functor}
Let $\mathfrak{M} = (\class{Q},\class{W},\class{R})$ be an hereditary abelian model structure.
\begin{enumerate}
\item  For each cofibrant $C \in \class{Q}$ and $n \geq 1$, the covariant Yoneda Ext functor $\Ext^n_{\cat{A}}(C,-) : \cat{A} \xrightarrow{} \textnormal{Ab}$ identifies left homotopic maps. In particular, it descends to an additive functor $\Ext^n_{\cat{A}}(C,-) : \textnormal{St}_{\omega}(\cat{A}) \xrightarrow{} \textnormal{Ab}$ by factoring through $\gamma_{\omega} : \cat{A} \xrightarrow{} \textnormal{St}_{\omega}(\cat{A})$.

\item For  each fibrant $F \in \class{R}$ and $n \geq 1$, the contravariant Yoneda Ext functor $\Ext^n_{\cat{A}}(-,F) : \cat{A} \xrightarrow{} \textnormal{Ab}$ identifies right homotopic maps. In particular, it descends to an additive functor $\Ext^n_{\cat{A}}(-,F) : \textnormal{St}_{\omega}(\cat{A}) \xrightarrow{} \textnormal{Ab}$  by factoring through $\gamma_{\omega} : \cat{A} \xrightarrow{} \textnormal{St}_{\omega}(\cat{A})$.
\end{enumerate}
\end{proposition}

\begin{proof}
We will only prove (1) as statement (2) is dual. We need to show that $\Ext^n_{\cat{A}}(C,f) = \Ext^n_{\cat{A}}(C, g)$ whenever $[f]_{\ell} = [g]_{\ell}$. (That is, whenever $f$ and $g$ are left homotopic, which by~\cite[Proposition~4.4]{gillespie-exact model structures}, is the case if and only if their difference factors through an object of $\class{R}_{\class{W}} := \class{W} \cap \class{R}$.)

We recall the Yoneda description of $\Ext^n_{\cat{A}}(C,-)$. For a given object $A$, the elements of $\Ext^n_{\cat{A}}(C,A)$ are equivalence classes of exact $n$-sequences of the form
$$\epsilon: \ \  0 \xrightarrow{} A \xrightarrow{} Y_n  \xrightarrow{} \cdots \xrightarrow{} Y_{2}  \xrightarrow{} Y_1 \xrightarrow{} C \xrightarrow{} 0.$$ As described in~\cite[pp.~79]{weibel}, the equivalence relation is \emph{generated} by the relation $\sim$, where $\epsilon' \sim \epsilon$ means there exists some commutative diagram of the form
$$\begin{CD}
\epsilon' : \ \ 0 @>>>  A @>>> Y'_n @>>> \cdots @>>> Y'_1 @>>> C  @>>> 0 \\
@. @|  @VVV @. @VVV @|  @.\\
\epsilon : \ \ 0 @>>> A @>>>  Y_n  @>>> \cdots @>>> Y_1 @>>> C  @>>> 0 \\
\end{CD}$$
As for morphisms, we recall that for a given $f : A \xrightarrow{} B$, then $$\Ext^n_{\cat{A}}(C,f) : \Ext^n_{\cat{A}}(C,A) \xrightarrow{} \Ext^n_{\cat{A}}(C,B)$$ is a group homomorphism defined by the pushout construction:
$$\begin{CD}
     0   @>>>  A    @>>>   Y_n @>>> Y_{n-1} @>>> \cdots   @>>>  Y_1 @>>> C    @>>>    0 \\
    @.        @VfVV     @VVV      @|  @. @| @| \\
    0   @>>>  B    @>>>   P    @>>> Y_{n-1} @>>> \cdots  @>>>  Y_1 @>>> C    @>>> 0   . \\
    \end{CD}
$$ 
That these constructions determine a well-defined additive functor $\Ext^n_{\cat{A}}(C,-) : \cat{A} \xrightarrow{} \textnormal{Ab}$ is well-known and so it remains to show $\Ext^n_{\cat{A}}(C,f) = 0$ whenever $f$ factors through an object of $\class{R}_{\class{W}}$. Our proof will rely on the fact that we can replace any representative of $\Ext^n_{\cat{A}}(C,A)$:
$$\epsilon: \ \  0 \xrightarrow{} A \xrightarrow{k} Y_n  \xrightarrow{f_n}  Y_{n-1} \xrightarrow{f_{n-1}} \cdots \xrightarrow{} Y_{2}  \xrightarrow{f_{2}} Y_{1} \xrightarrow{f_1} C \xrightarrow{} 0,$$ with an equivalent $n$-sequence 
$$\epsilon': \ \  0 \xrightarrow{} A \xrightarrow{k'} P_n  \xrightarrow{t_n}  Q_{n-1} \xrightarrow{t_{n-1}} \cdots \xrightarrow{} Q_{2}  \xrightarrow{t_{2}} Q_{1} \xrightarrow{t_1} C \xrightarrow{} 0, $$ such that for each $i = 1,2,\cdots,n-1$, the object $L_i :=\ker{t_i} \in \class{Q}$. 
So lets first prove this.

We start on the right end of $\epsilon$, considering the exact 2-sequence shown 
$$0 \xrightarrow{} \ker{f_{2}} \xrightarrow{} Y_{2} \xrightarrow{f_2} Y_{1} \xrightarrow{f_1} C \xrightarrow{} 0.$$ 
Since $(\class{Q}, \class{R}_{\class{W}} )$ has enough projectives, there is an epimorphism $p : Q_1 \twoheadrightarrow Y_1$ with $Q_1 \in \class{Q}$.  Letting $P_2$ denote the pullback of $Y_2 \xrightarrow{f_2} Y_1 \xleftarrow{p} Q_1$, one constructs a morphism of exact 2-sequences 
$$\begin{CD}
0 @>>>  \ker{f_{2}} @>>> P_{2} @>f'_{2}>> Q_{1} @>t_{1}>> C  @>>> 0 \\
@. @|  @Vp'VV @VpVV @|  @.\\
0 @>>> \ker{f_{2}} @>>>  Y_{2}  @>f_{2}>> Y_{1} @>f_{1}>> C  @>>> 0. \\
\end{CD}$$
But then $L_{1} := \ker{t_{1}}$ is also cofibrant because the model structure is hereditary. This morphism of exact 2-sequences can be extended to a morphism of exact $n$-sequences, by ``pasting'' the top row, at $\ker{f_{2}}$,  together with the rest of $\epsilon$, yielding the (Yoneda equivalent) exact $n$-sequence shown:
$$0 \xrightarrow{} A \xrightarrow{k} Y_n  \xrightarrow{f_n} Y_{n-1} \xrightarrow{f_{n-1}} \cdots \xrightarrow{f_4} Y_{3}  \xrightarrow{f'_{3}}  P_{2}  \xrightarrow{f'_{2}} Q_{1} \xrightarrow{t_{1}} C \xrightarrow{} 0.$$

Next, we focus on the portion of this new $n$-sequence shown below:
$$0 \xrightarrow{} \ker{f'_{3}} \xrightarrow{} Y_{3} \xrightarrow{f'_3} P_{2} \xrightarrow{f'_{2}} L_{1} \xrightarrow{} 0.$$ Recalling that $L_{1}$ is cofibrant, we repeat the same procedure, obtaining another Yoneda equivalent exact $n$-sequence as shown:
$$0 \xrightarrow{} A \xrightarrow{k} Y_n  \xrightarrow{f_n} Y_{n-1} \xrightarrow{f_{n-1}} \cdots \xrightarrow{f_5} Y_{4}  \xrightarrow{f'_{4}} P_{3}  \xrightarrow{f''_{3}}  Q_{2}  \xrightarrow{t_{2}} Q_{1} \xrightarrow{t_{1}} C \xrightarrow{} 0,$$
this one also having $L_{2} := \ker{t_{2}}$ cofibrant.

In this way, we can continue the process, from right to left, finally obtaining the desired (Yoneda equivalent) exact $n$-sequence
$$\epsilon': \ \  0 \xrightarrow{} A \xrightarrow{k'} P_{n}  \xrightarrow{f''_n} Q_{n-1} \xrightarrow{t_{n-1}} Q_{n-2} \xrightarrow{t_{n-2}} \cdots \xrightarrow{} Q_{2}  \xrightarrow{t_{2}} Q_{1} \xrightarrow{t_{1}} C \xrightarrow{} 0 $$ having $L_i :=\ker{t_i}$ cofibrant for each $i = 1,2,\cdots,n-1$.

So now, finally, we are able to show $\Ext^n_{\cat{A}}(C,f) = 0$ whenever $f : A \xrightarrow{} B$ factors as $f = \beta \alpha$ through an object $W \in \class{R}_{\class{W}}$. In this case, since $\Ext^1_{\cat{A}}(L_{n-1},W)=0$,  the morphism $A \xrightarrow{\alpha} W$ extends over $A \xrightarrow{k'} P_{n}$.
The Homotopy Lemma~\cite[Lemma~3.2]{gillespie-stable-cats-cotorsion-pairs} now applies and implies that the pushed-out $n$-sequence in the bottom row of the diagram below represents 0 in $\Ext^n_{\cat{A}}(C,B)$.
$$\begin{CD}
     0   @>>>  A    @>k'>>   P_{n} @>t_n>> Q_{n-1} @>t_{n-1}>> \cdots   @>t_1>> C    @>>>    0 \\
    @.        @VfVV     @VVV      @|  @. @|  \\
    0   @>>>  B    @>>>   P    @>>> Q_{n-1} @>t_{n-1}>> \cdots  @>t_1>> C    @>>> 0   . \\
    \end{CD}
$$ 
\end{proof}

With Proposition~\ref{descending-Ext-functor} in hand, we make the following definition and obtain the next theorem. Note the analogy to Definition~\ref{def-Ho-Ext} and Theorem~\ref{them-Ho-Ext}. However, we are now restricted to only $n \geq 1$.

\begin{definition}\label{def-Ho-Yoneda-Ext}
$\underline{\Ext}^n_{\,\textnormal{Ho}(\mathfrak{M})}(A,B) := \Ext^n_{\cat{A}}(RQA,RQB)$, for  $n \geq 1$.
\end{definition}
 %However, they may be viewed as an extension of the $n=0$ case: $$\textnormal{Ho}(\mathfrak{M})(A,B) := \underline{\Hom}_{\,\omega}(RQA,RQB).$$

We again recall that there is a canonical functor $\gamma_{\mathfrak{M}} : \cat{A} \xrightarrow{} \textnormal{Ho}(\mathfrak{M})$. It is the identity on objects, and takes a morphism $f : A \xrightarrow{} B$ to the homotopy class $[R(Q(f))]_{\omega} \in \underline{\Hom}_{\,\omega}(RQA,RQB)$.

\begin{theorem}\label{them-Ho-Yoneda-Ext}
For each $n \geq 1$, $\underline{\Ext}^n_{\,\textnormal{Ho}(\mathfrak{M})}(A,B)$ defines a functor with the following properties.  
\begin{enumerate}
\item $\underline{\Ext}^n_{\,\textnormal{Ho}(\mathfrak{M})}(A,-): \cat{A} \xrightarrow{} \textbf{Ab}$ is a covariant additive functor and short exact sequences in $\cat{A}$ are sent to long exact cohomology sequences of $\underline{\Ext}^n_{\,\textnormal{Ho}(\mathfrak{M})}$ groups. $\underline{\Ext}^n_{\,\textnormal{Ho}(\mathfrak{M})}(A,-)$ descends via $\gamma_{\mathfrak{M}} : \cat{A} \xrightarrow{} \textnormal{Ho}(\mathfrak{M})$ to a well-defined functor on  $\textnormal{Ho}(\mathfrak{M})$.
\item $\underline{\Ext}^n_{\,\textnormal{Ho}(\mathfrak{M})}(-,B) : \cat{A} \xrightarrow{} \textbf{Ab}$ is a contravariant additive functor and short exact sequences in $\cat{A}$ are sent to long exact cohomology sequences of $\underline{\Ext}^n_{\,\textnormal{Ho}(\mathfrak{M})}$ groups. $\underline{\Ext}^n_{\,\textnormal{Ho}(\mathfrak{M})}(-,B)$ descends via $\gamma_{\mathfrak{M}} : \cat{A} \xrightarrow{} \textnormal{Ho}(\mathfrak{M})$ to a well-defined functor on  $\textnormal{Ho}(\mathfrak{M})$.
\item We have the following isomorphisms, natural in both $A$ and $B$:
 $$\underline{\Ext}^n_{\,\textnormal{Ho}(\mathfrak{M})}(A,B) \cong \Ext^n_{\cat{A}}(QA,RB).$$
\item We have the following isomorphisms, natural in both $A$ and $B$: 
$$\textnormal{Ho}(\mathfrak{M})(\Omega^n A,B) \cong \Ext^n_{\cat{A}}(QA,RB) \cong \textnormal{Ho}(\mathfrak{M})(A,\Sigma^n B).$$ 
\end{enumerate}
\end{theorem}

\begin{proof}
With the help of Proposition~\ref{descending-Ext-functor}, the proofs of (1)--(3) are very similar to the proofs of (1)--(3) of Theorem~\ref{them-Ho-Ext}.  It remains to prove (4) and 
we will now show $\textnormal{Ho}(\mathfrak{M})(\Omega A,B) \cong \Ext^1_{\cat{A}}(QA,RB)$. As shown in Proposition~\ref{prop-suspension functor}, given an object $A$ we compute $\Omega A$ by taking a short exact sequence $\Omega A \xrightarrow{}  W_A \xrightarrow{} A \xrightarrow{} 0$ with $W_A \in \class{W}$. (A potentially different $\Omega A$ resulting from a different short exact sequence will be canonically isomorphic, in $\textnormal{Ho}(\mathfrak{M})$). From the generalized horseshoe lemma, see~\cite[Lemma~1.4.4]{becker} or~\cite{gillespie-stable-cats-cotorsion-pairs}, we can find a cofibrant replacement sequence as in  the top row below: 
$$\begin{CD}
    0   @>>>  Q\Omega A    @>>>   QW_A    @>>>  QA    @>>>  0  \\
    @.        @VVV     @VVV      @VVV   @.\\
    0   @>>>  \Omega A    @>>>   W_A    @>>>  A   @>>>  0  \\
    \end{CD}$$
Since $\class{W}$ is closed under extensions we note that $QW_A \in \class{Q}_{\class{W}}$, the class of trivially cofibrant objects. Now given the other object $B$, we apply $\Hom_{\class{A}}(-,RB)$ to the top row and it gives us a homomorphism $$\delta : \Hom_{\class{A}}(Q\Omega A,RB) \xrightarrow{} \Ext^1_{\cat{A}}(QA,RB).$$ In the Yoneda Ext description, $\delta$ is defined via ``pushout'', as indicated below:
\begin{equation}\label{diag-pushout} 
\begin{CD}
 \ \ \ \ \ \ \ \ 0   @>>>  Q\Omega A    @>>>   QW_A    @>>>  QA    @>>>  0  \\
    @.        @VfVV     @VVV      @|   @.\\
 \delta(f) :   0   @>>> RB    @>>>   P   @>>>   Q A    @>>>  0  \\
    \end{CD}
\end{equation}
We can prove directly that $\delta$ is onto. Indeed given any short exact sequence as in the bottom row below, we use that $\Ext^1_{\cat{A}}(QW_A,RB) = 0$ to construct a morphism of short exact sequences as shown.
$$\begin{CD}
    0   @>>>  Q\Omega A    @>>>   QW_A    @>>>  QA    @>>>  0  \\
    @.        @VtVV     @VVV      @|   @.\\
    0   @>>> RB    @>>>   Z    @>>>   Q A    @>>>  0  \\
    \end{CD}$$ 
By~\cite[Proposition~2.12]{buhler-exact categories} the left square is a pushout, proving that the bottom row is $\delta(t)$ as desired.

\

\noindent \underline{Claim}: $\ker{\delta} = \{\,f \in \Hom_{\cat{A}}(Q\Omega A,RB)  \, | \, f \sim^{\omega} 0  \,\}$ 

\

To prove the claim, suppose $\delta(f) = 0$. It means that there exists a lift  $QA \xrightarrow{} P$ (or \emph{section}) in the pushout diagram~(\ref{diag-pushout}) making the lower right triangle commute. But by the Homotopy Lemma~\cite[Lemma~3.2]{gillespie-stable-cats-cotorsion-pairs}, this is equivalent to a morphism $QW_A \xrightarrow{} RB$ making the upper left triangle of diagram~(\ref{diag-pushout}) commute. This proves $f$ factors through an object of $\class{Q}_{\class{W}}$. But since the source object $Q\Omega A$ is cofibrant and the target object $RB$ is fibrant, we conclude by~\cite[Proposition~4.4(5)]{gillespie-exact model structures} that $f$ actually factors through an object of $\omega$. This proves $\subseteq$. To prove $\supseteq$, suppose $f$ factors as $Q\Omega A \xrightarrow{\alpha} W \xrightarrow{\beta} RB$ where $W \in \omega$. Then applying $\Hom_{\class{A}}(-,W)$ to the top row of diagram~(\ref{diag-pushout}), and using $\Ext^1_{\cat{A}}(QA,W) = 0$, we see that $\alpha$ extends through $Q\Omega A  \xrightarrow{} QW_A$. Composing the new map with $\beta$, the Homotopy Lemma now allows us to conclude $\delta(f)$ splits, so $f \in \ker{\delta}$. This completes the proof of the Claim.

 Thus $\delta$ descends to an isomorphism $$\underline{\Hom}_{\,\omega}(Q\Omega A,RB) \xrightarrow{\bar{\delta}} \Ext^1_{\cat{A}}(QA,RB).$$
The result for $n=1$ now follows by composing with the natural isomorphism $$\textnormal{Ho}(\mathfrak{M})(A,B) \cong \underline{\Hom}_{\,\omega}(QA,RB).$$ 
For $n>1$, we may use an inductive dimension shifting argument. For example, from what we just proved we have $\textnormal{Ho}(\mathfrak{M})(\Omega^2 A,B) \cong \Ext^1_{\cat{A}}(Q\Omega A,RB)$. But applying  $\Hom_{\cat{A}}(-,RB)$ to the short exact sequence  $0   \xrightarrow{}  Q\Omega A     \xrightarrow{}   QW_A     \xrightarrow{}  QA    \xrightarrow{}   0$ we deduce $\Ext^1_{\cat{A}}(Q\Omega A,RB) \cong \Ext^2_{\cat{A}}(QA,RB)$. Note that this dimension shifting argument relies on the fact that $(\class{Q}_{\class{W}},\class{R})$ is an hereditary cotorsion pair because we need $\Ext^i_{\cat{A}}(QW_A,RB) = 0$ for all $i \geq 1$.

We have shown $\textnormal{Ho}(\mathfrak{M})(\Omega^n A,B) \cong \Ext^n_{\cat{A}}(QA,RB)$ and a dual argument will prove $\textnormal{Ho}(\mathfrak{M})(A,\Sigma^n B) \cong \Ext^n_{\cat{A}}(QA,RB)$. Of course this also must be automatic since $\Sigma$ and $\Omega$ are inverse autoequivalences on $\textnormal{Ho}(\mathfrak{M})$.
\end{proof}

From parts (3) and (4) of Theorem~\ref{them-Ho-Yoneda-Ext}, along with Corollary~\ref{cor-Ho-Ext}, we obtain the following corollary.

\begin{corollary}\label{cor-Ext}
For all objects $A$ and $B$ and each $n \geq 1$ we have natural isomorphisms 
$$\Ext^n_{\textnormal{Ho}(\mathfrak{M})}(A,B) \cong \underline{\Ext}^n_{\,\textnormal{Ho}(\mathfrak{M})}(A,B)  \cong \Ext^n_{\cat{A}}(QA,RB).$$
\end{corollary}

Note that while the above description of $\Ext^n_{\textnormal{Ho}(\mathfrak{M})}(A,B)$ only holds for positive integers we may use Corollary~\ref{cor-dim-shift} to express any $\Ext^n_{\textnormal{Ho}(\mathfrak{M})}$ group as an $\Ext^1_{\cat{A}}$ group. For example, 
$$\Ext^{-1}_{\textnormal{Ho}(\mathfrak{M})}(A,B) \cong \Ext^0_{\textnormal{Ho}(\mathfrak{M})}(\Sigma A,B)\cong \Ext^1_{\textnormal{Ho}(\mathfrak{M})}(\Sigma^2 A,B) \cong \Ext^1_{\cat{A}}(Q\Sigma^2 A, RB).$$ In particular, taking $n=0$ and $m=1$ in Corollaries~\ref{cor-Ho-Ext} and~\ref{cor-dim-shift}, then applying Theorem~\ref{them-Ho-Yoneda-Ext}~(4) along with Proposition~\ref{prop-functors-commute} yields the following result describing morphism sets in the homotopy category as $\Ext^1_{\cat{A}}$ groups on the ground category $\cat{A}$.

\begin{corollary}\label{cor-Ho-sets}
For all objects $A$ and $B$ we have natural isomorphisms 
$$\Ext^1_{\cat{A}}(\Sigma QA,RB)   \cong \textnormal{Ho}(\mathfrak{M})(A,B)  \cong \Ext^1_{\cat{A}}(QA,\Omega RB).$$
\end{corollary}

\begin{example}\label{example-K(A)}
Let $\cat{A}$ be any additive category, including perhaps an exact or abelian category. Let $\cha{A}_{dw}$ denote the associated chain complex category along with the degreewise split exact structure. Denote the Yoneda Ext groups by $\Ext^n_{dw}(X,Y)$. They are equivalence classes of $n$-fold exact sequences of chain complexes, obtained by the splicing together of degreewise split short exact sequences. $\cha{A}_{dw}$ is a well-known Frobenius category and the contractible complexes serve as the projective-injective objects. It means we have a Hovey triple $\mathfrak{M}$ = (All, Contractible complexes, All). Its homotopy category is precisely $K(\cat{A})$, the usual chain homotopy category of complexes. In light of Example~\ref{example-suspensions} from Appendix~\ref{sec-suspension and loop}, the isomorphisms at the end of Corollary~\ref{cor-Ho-sets} recover the well-known isomorphisms:
$$\Ext^1_{dw}(\Sigma X,Y) \cong K(\cat{A})(X,Y)  \cong \Ext^1_{dw}(X,\Sigma^{-1} Y).$$
The statement of Theorem~\ref{them-Ho-Yoneda-Ext}~(4)  becomes the more general variation
$$K(\cat{A})(\Sigma^{-n}X,Y) \cong \Ext^n_{dw}(X,Y) \cong K(\cat{A})(X,\Sigma^n Y).$$
Similar statements are recovered for any Frobenius category. So for example, taking a field $k$ and a finite group $G$, then the category $kG$-Mod of modules over the group algebra $kG$ is a Frobenius category. The analogous statements recover basic facts of group cohomology.
\end{example}

%%%%%%%%%%%%%%%%%%%%%%%%%%%%%%%%%%%%%%%%%%%%%%%%%%%%%%%%%

\section{The functors $\Tor_n^{\textnormal{Ho}(\mathfrak{M})}$}\label{section-Tor}

In this section we use canonical resolutions to define Tor functors, assuming that all trivially cofibrant objects are flat. To this end we suppose that we have a covariant additive functor $- \tensor - : \cat{A} \times \cat{A}'  \xrightarrow{} \textbf{Ab}$, called a \emph{tensor product}, defined on abelian (or exact) categories $\cat{A}$ and $\cat{A}'$.  Recall, $\textbf{Ab}$ is the category of abelian groups. For each object $A \in \cat{A}$, the functor $A \tensor -$ is assumed to be right exact in the sense that it takes short exact sequences to right exact sequences in  $\textbf{Ab}$. We say that $A$ is \emph{flat} if $A \tensor -$ is an exact functor. Similarly,  each $- \tensor A'$ is assumed right exact and we say $A'$ is flat if it is exact. 

We keep our usual running assumption that the model structure $\mathfrak{M} = (\class{Q},\class{W},\class{R})$ on $\cat{A}$ is hereditary. But in addition, we assume throughout this section that all objects in $\class{Q}\cap\class{W}$ are flat.  Similarly, we assume $\cat{A}'$ has an hereditary model structure $\mathfrak{M}' = (\class{Q}',\class{W}',\class{R}')$ for which all objects in $\class{Q}' \cap \class{W}'$ are flat. Note that these assumptions imply that each object of $\cat{A}$, or $\cat{A}'$, can be represented as a quotient of a flat object. 

\begin{lemma}\label{lemma-purity}
Any short exact sequence in $\cat{A}$, or $\cat{A}'$, ending with a flat object is a \emph{pure exact sequence}. That is, it remains exact after being tensored by any object of the other category. 
\end{lemma}

\begin{proof}
A standard formal argument, see for example~\cite[Lemma~XVI.3.3]{lang}, will work in this setting.  
\end{proof}

\begin{definition}\label{def-lTor-rTor}
For objects $A \in \cat{A}$ and $B \in \cat{A}'$, we define $$\ell\Tor_n^{\mathfrak{M}}(A,B) := H_n[W_A \tensor B],$$ where $W_A$ is an $\mathfrak{M}$-resolution of $A$.  Similarly we define
$$r\Tor_n^{\mathfrak{M}}(A,B) := H_n[A \tensor W_B],$$ where $W_B$ is an $\mathfrak{M}'$-resolution of $B$. 
Different choices for  $W_A$ or $W_B$ yield canonical isomorphisms of abelian groups.
\end{definition}

Let us prove the statement that  $\ell\Tor_n^{\mathfrak{M}}(A,B)$ is well-defined (up to a canonical isomorphism), regardless of the choice of resolution $W_A$. Indeed if $W_A$ and $W'_A$ are two $\mathfrak{M}$-resolutions of $A$, then by Corollary~\ref{cor-K(A)} there is a canonical isomorphism $[\alpha] : W_A \xrightarrow{} W'_A$ in $K(\cat{A})$. For any object $B \in \cat{A}'$, since $- \tensor B$ is a covariant additive functor, it takes $\alpha$ to a chain homotopy equivalence $W_A \tensor B \xrightarrow{\alpha \tensor B} W'_A \tensor B$ of chain complexes of abelian groups. It follows that we have a canonical isomorphism: 
$$H_n[W_A \tensor B] \xrightarrow{H_n([\alpha \tensor B])} H_n[W'_A \tensor B].$$

\begin{theorem}\label{them-lTor}
Each $\ell\Tor_n^{\mathfrak{M}}(A,B)$ is a covariant additive functor in both $A$ and $B$ and satisfies the following properties:
\begin{enumerate}
\item If $A \in \cat{A}$ is cofibrant, then short exact sequences in $\cat{A}'$ are sent to long exact homology sequences of $\ell\Tor_n^{\mathfrak{M}}$ groups. Moreover, for any $A$, the functors $\ell\Tor^n_{\mathfrak{M}}(A,-)$ identify right homotopic maps. In particular each factors through $\gamma' : \cat{A}' \xrightarrow{} \textnormal{St}_{\omega'}(\cat{A}')$. 
\item For any $B \in \cat{A}'$, each cofibration in $\cat{A}$ induces a long exact homology sequence of $\ell\Tor_n^{\mathfrak{M}}(-,B)$ groups. Moreover, for any $B$, each $\ell\Tor_n^{\mathfrak{M}}(-,B)$ factors through $\gamma : \cat{A} \xrightarrow{} \textnormal{St}_{\omega}(\cat{A})$.
\end{enumerate}
\end{theorem}

\begin{proof}
For any fixed resolution $W_A$ of $A$ we have the functor $H_n[W_A \tensor -]$ which is covariant and additive and so $\ell\Tor_n^{\mathfrak{M}}(A,-) := H_n[W_A \tensor -]$ is such a functor. For a different choice of resolution $W'_A$ then going back to the paragraph after Definition~\ref{def-lTor-rTor}, one can verify that the $[\alpha \tensor B]$ assemble to provide a natural isomorphism of functors $$\{H_n([\alpha \tensor B])\} : H_n[W'_A \tensor -] \xrightarrow{} H_n[W_A \tensor -].$$ Thus $\ell\Tor_n^{\mathfrak{M}}(A,-)$ is a functor that is well-defined up to canonical isomorphism.
 
 We now go on to prove statement~(1). So suppose $A \in \cat{A}$ is cofibrant and $$0 \xrightarrow{} B' \xrightarrow{f} B \xrightarrow{g} B'' \xrightarrow{} 0$$ is a short exact sequence. By Lemma~\ref{lemma-can-resolutions-cofibrant}~(4), we know that all components of $W_A$ are (trivially) cofibrant, and therefore flat. So applying $W_A \tensor -$ gives us the following short exact sequence
 $$0 \xrightarrow{} W_A \tensor B' \xrightarrow{} W_A \tensor B \xrightarrow{} W_A \tensor B'' \xrightarrow{} 0.$$ Now applying homology $H_n$, the fundamental lemma of homological algebra provides the long exact sequence of $\ell\Tor_n^{\mathfrak{M}}$ groups. 
 
Next, we show that for any $A \in \cat{A}$, the functors $\ell\Tor_n^{\mathfrak{M}}(A,-)$ identify right homotopic maps. We recall that two maps are right homotopic, written $f \sim^{r} g$, if and only if $g-f$ factors through a trivially cofibrant object. We wish to show that $\ell\Tor_n^{\mathfrak{M}}(A,f) =  \ell\Tor_n^{\mathfrak{M}}(A,g)$ for such maps. But since  $\ell\Tor_n^{\mathfrak{M}}(A,-)$ is an additive functor it is enough to show  $\ell\Tor_n^{\mathfrak{M}}(A,h) = 0$ whenever $h \sim^{r} 0$. So suppose $h = \beta\alpha$ where $B_1 \xrightarrow{\alpha} Q \xrightarrow{\beta} B_2$ and $Q \in \class{Q}'\cap{\class{W}'}$.
Then we get $$W_A \tensor B_1 \xrightarrow{\alpha \tensor B} W_A \tensor Q \xrightarrow{\beta \tensor B} W_A \tensor B_2.$$ But since $Q$ is flat the complex $W_A \tensor Q$ is acyclic, and so $\ell\Tor_n^{\mathfrak{M}}(A,Q) = 0$ for all $n$. This proves that $\ell\Tor_n^{\mathfrak{M}}(A,h) :  \ell\Tor_n^{\mathfrak{M}}(A,B_1)  \xrightarrow{}  \ell\Tor_n^{\mathfrak{M}}(A,B_2)$ is 0 for all $n$.

We now turn to show that $\ell\Tor_n^{\mathfrak{M}}(A,B)$ is a functor in variable $A$ and prove the statements in~(2). First, Corollary~\ref{cor-K(A)} tells us that $\mathfrak{M}$-resolutions provide a functor, $\cat{A} \xrightarrow{} K(\cat{A})$, well-defined on objects up to a canonical isomorphism, and that it factors as $\cat{A} \xrightarrow{\gamma_{\omega}} \textnormal{St}_{\omega}(\cat{A}) \xrightarrow{} K(\cat{A})$.  For a fixed $B$, since $- \tensor B$ is a covariant additive functor it induces a functor $- \tensor B : K(\cat{A}) \xrightarrow{} K(\textbf{Ab})$, where here $K(\textbf{Ab})$ is the homotopy category of chain complexes of abelian groups. Finally, composing all these functors with homology gives us a covariant additive functor
$$\cat{A} \xrightarrow{\gamma_{\omega}} \textnormal{St}_{\omega}(\cat{A}) \xrightarrow{} K(\cat{A})  \xrightarrow{- \tensor B} K(\textbf{Ab})   \xrightarrow{H_n} \textbf{Ab}, $$
which is precisely the functor $\ell\Tor_n^{\mathfrak{M}}(-,B)$.

Finally, suppose we have a short exact sequence $$\class{E} : 0 \xrightarrow{} A \xrightarrow{f} A' \xrightarrow{g} C \xrightarrow{} 0$$ with $C$ cofibrant. Then using the Horseshoe Lemma~\ref{lemma-horseshoe-lemma} we may extend $\class{E}$ to a short exact sequence of $\mathfrak{M}$-resolutions $$0 \xrightarrow{} W_A \xrightarrow{\{f_n\}} W_{A'} \xrightarrow{\{g_n\}} W_{C} \xrightarrow{} 0.$$ 
Since $C$ is cofibrant, all components of $W_C$ are trivially cofibrant by Lemma~\ref{lemma-can-resolutions-cofibrant}~(4). So by Lemma~\ref{lemma-purity}, applying $- \tensor B$ gives us another short exact sequence
 $$0 \xrightarrow{} W_A \tensor B \xrightarrow{\{f_n \tensor B \}} W_{A'} \tensor B \xrightarrow{\{g_n \tensor B\}} W_C \tensor B \xrightarrow{} 0.$$ Finally, applying homology $H_n$, the fundamental lemma of homological algebra provides the long exact sequence of $\ell\Tor_n^{\mathfrak{M}}$ groups. 
\end{proof}

We state the right version now too. 

\begin{theorem}\label{them-rTor}
Each $r\Tor_n^{\mathfrak{M}}(A,B)$ is a covariant additive functor in both $A$ and $B$ and satisfies the following properties:
\begin{enumerate}
\item If $B \in \cat{A}'$ is cofibrant, then short exact sequences in $\cat{A}$ are sent to long exact homology sequences of $r\Tor_n^{\mathfrak{M}}$ groups. Moreover, for any $B$, the functors $r\Tor^n_{\mathfrak{M}}(-,B)$ identify right homotopic maps. In particular each factors through $\gamma : \cat{A} \xrightarrow{} \textnormal{St}_{\omega}(\cat{A})$. 
\item For any $A \in \cat{A}$, each cofibration in $\cat{A}'$ induces a long exact homology sequence of $\ell\Tor_n^{\mathfrak{M}}(A,-)$ groups. Moreover, for any $A$, each $r\Tor_n^{\mathfrak{M}}(A,-)$ factors through $\gamma' : \cat{A}' \xrightarrow{} \textnormal{St}_{\omega'}(\cat{A}')$.
\end{enumerate}
\end{theorem}

We also have the following easy lemmas concerning the vanishing of the functors $\ell\Tor_n^{\mathfrak{M}}$ and  $r\Tor_n^{\mathfrak{M}}$.

\begin{lemma}\label{lemma-lTor-vanishing}
If either $A$ or $B$ is trivially cofibrant, then for all $n \in \Z$, we have both  $\ell\Tor_n^{\mathfrak{M}}(A,B) = 0$ and  $r\Tor_n^{\mathfrak{M}}(A,B) = 0$. 
\end{lemma}

\begin{proof}
It follows immediately from Lemma~\ref{lemma-can-resolutions-cofibrant}~(1)+(2) along with Lemma~\ref{lemma-purity}.
\end{proof}

$\Tor_n^{\mathfrak{M}}(A,B)$ is balanced in the case that $A$ and $B$ are each cofibrant. 

\begin{proposition}\label{prop-left-right-Tor}
If $A$ and $B$ are each cofibrant then we have natural isomorphisms $\ell\Tor_n^{\mathfrak{M}}(A,B) \cong r\Tor_n^{\mathfrak{M}}(A,B)$  for all $n \in \Z$.
\end{proposition}

\begin{proof}
One can imitate the usual proof that Tor is balanced by flat modules. See~\cite[Theorem~2.7.2]{weibel}. The key ingredient is to note that all components of $W_A$ and $W_B$ will be trivially cofibrant, and hence flat. So each may be truncated at the cycles in any degree we wish; for example,
$$\cdots \xrightarrow{} W_{n+2} \xrightarrow{} W_{n+1} \xrightarrow{} Z_nW \xrightarrow{} 0$$   
becomes a flat resolution of $Z_nW$. Hence the \cite[Acyclic Assembly Lemma Theorem~2.7.3]{weibel} will apply to conclude that $\ell\Tor_k^{\mathfrak{M}}(A,B) \cong r\Tor_k^{\mathfrak{M}}(A,B)$  for all $k > n$, and hence for all $k \in \Z$.
\end{proof}

Because of the natural isomorphism  $\ell\Tor_n^{\mathfrak{M}}(A,B) \cong r\Tor_n^{\mathfrak{M}}(A,B)$ in Proposition~\ref{prop-left-right-Tor} we may, in the case that $A$ and $B$ are each is cofibrant, simply denote this group by $\Tor_n^{\mathfrak{M}}(A,B)$, with the realization that it may be computed by either a canonical resolution of $A$, or, a canonical resolution of $B$. With this observation we make the following definition. 

\begin{definition}\label{def-Ho-Tor}
$\Tor_n^{\textnormal{Ho}(\mathfrak{M})}(A,B) := \Tor_n^{\mathfrak{M}}(RQA,RQB)$.
\end{definition}

The main properties of $\Tor_n^{\textnormal{Ho}(\mathfrak{M})}(A,B) $ appear in the following theorem. 

\begin{theorem}\label{them-Ho-Tor}
Each $\Tor_n^{\textnormal{Ho}(\mathfrak{M})}(A,B)$ defines a functor with the following properties.  
\begin{enumerate}
\item $\Tor_n^{\textnormal{Ho}(\mathfrak{M})}(A,-) : \cat{A}' \xrightarrow{} \textbf{Ab}$ is a covariant additive functor and short exact sequences in $\cat{A}'$ are sent to long exact homology sequences of $\Tor_n^{\textnormal{Ho}(\mathfrak{M})}$ groups. $\Tor_n^{\textnormal{Ho}(\mathfrak{M})}(A,-)$ descends via $\gamma' : \cat{A}' \xrightarrow{} \textnormal{Ho}(\cat{A}')$ to a well-defined functor on  $\textnormal{Ho}(\cat{A}')$.
\item $\Tor_n^{\textnormal{Ho}(\mathfrak{M})}(-,B) : \cat{A} \xrightarrow{} \textbf{Ab}$ is a covariant additive functor and short exact sequences in $\cat{A}$ are sent to long exact homology sequences of $\Tor_n^{\textnormal{Ho}(\mathfrak{M})}$ groups. $\Tor_n^{\textnormal{Ho}(\mathfrak{M})}(-,B)$
descends via $\gamma_{\cat{A}} : \cat{A} \xrightarrow{} \textnormal{Ho}(\cat{A})$ to a well-defined functor on  $\textnormal{Ho}(\cat{A})$.
\item We have the following isomorphisms, natural in both $A$ and $B$:
 $$\Tor_n^{\textnormal{Ho}(\mathfrak{M})}(A,B) \cong \Tor_n^{\mathfrak{M}}(QA,QB).$$ In particular, the group $\Tor_n^{\textnormal{Ho}(\mathfrak{M})}(A,B)$ may be computed by either applying $H^n[- \tensor QB]$ to a canonical resolution of $A$, or, by applying $H_n[QA \tensor -]$ to a canonical resolution of $B$.
 \item Suppose that our categories $\cat{A}$ and $\cat{A}'$ have enough projectives. Then the usual construction of left derived functors via projective resolutions yields bifunctors $\Tor_n(-,-)$ for $n \geq 0$. We have $$\Tor_n^{\textnormal{Ho}(\mathfrak{M})}(A,B) \cong \Tor_n(QA,QB)$$ for all $n \geq 1$.
\end{enumerate}
\end{theorem}

\begin{proof}
The reader can formulate a proof by imitating the proof of Theorem~\ref{them-Ho-Ext}. To prove~(4) assume we have enough projective. So for any $B$ we have left derived functors $\Tor_n(-, B)$ of $- \tensor B$. We note that if $C$ is a trivially cofibrant object then $\Tor_n(C, B) = 0$ for $n \geq 1$. Indeed any projective resolution of $C$ will have all kernels trivially cofibrant since $\class{Q}_{\class{W}}$ contains all projectives and by the hereditary hypothesis. It follows from Lemma~\ref{lemma-purity} that $\class{P} \tensor B \xrightarrow{} C \tensor B \xrightarrow{} 0$ remains exact for any projective resolution $\class{P} \xrightarrow{} C \xrightarrow{} 0$. This proves $\Tor_n(C, B) = 0$ for $n \geq 1$. It now follows from~\cite[Dual of Theorem~XX.6.2]{lang} that the usual left derived functors $\Tor_n(A, B) \,  (n \geq 0)$ may be computed, on any object $A$, via any resolution of $A$ by trivially cofibrant objects. So since $\Tor_n^{\textnormal{Ho}(\mathfrak{M})}(A,B)$ can be computed by applying $H_n[- \tensor QB]$ to any canonical resolution of $A$, and since any canonical resolution of $A$ will provide a resolution of $QA$ by trivially cofibrant objects, we conclude $\Tor_n^{\textnormal{Ho}(\mathfrak{M})}(A,B) \cong  \Tor_n(QA,QB)$ for $n \geq 1$.
\end{proof}

%%%%%%%%%%%%%%%%%%%%%%%%%%%%%%%%%%%%%%%%%%%%%%%%%%%%%%%%

\section{Canonical resolutions are cofibrant replacements}\label{sec-becker}

Hanno Becker shows in~\cite{becker-realization-functor} that if $\mathfrak{M} = (\class{Q},\class{W},\class{R})$ is a cofibrantly generated abelian model structure on a Grothendieck category $\cat{A}$, then $\mathfrak{M}$ will lift to several Quillen equivalent model structures on $\cha{A}$. We don't know whether or not these constructions can be done universally, for \emph{any} hereditary abelian model structure on \emph{any} abelian category $\cat{A}$. However the work in this paper certainly relates to the Becker's model structures and this section is aimed to make explicit how our canonical resolutions give rise to cofibrant replacements in Becker's model structures. 

Unless stated otherwise $\mathfrak{M} = (\class{Q},\class{W},\class{R})$ once again denotes any hereditary abelian model structure on any abelian (or exact) category $\cat{A}$. Throughout this section, we will let $\class{C}$ denote the class of all exact (acyclic) chain complexes with all cycles cofibrant but with all components trivially cofibrant. In the author's notation from~\cite{gillespie} and~\cite{gillespie-degreewise-model-strucs}, it means $\class{C} := \tilclass{Q} \cap dw\widetilde{\class{Q}_\class{W}}$. On the other hand, we let $\class{F}$ denote the class of all exact chain complexes with all cycles fibrant but with all components trivially fibrant. That is, $\class{F} := \tilclass{R} \cap dw\widetilde{\class{R}_\class{W}}$.

\subsection{Special $\class{C}$-precovers of sphere complexes} 
Let $A$ be any object of $\cat{A}$, and let $W_A$ be any $\mathfrak{M}$-resolution. So  
$$W_A \equiv (W_{n\geq0} \xrightarrow{\epsilon_A} A \xrightarrow{\eta_A} W_{n\leq-1}),$$
where $W_{n\geq0} \xrightarrow{\epsilon_A} A$ is a resolution constructed by using enough projectives of $(\class{Q}_{\class{W}},\class{R})$, and $A \xrightarrow{\eta_A} W_{n\leq-1}$ is a coresolution constructed by using enough injectives of $(\class{Q},\class{R}_{\class{W}})$. Writing $d_1 = e_1\eta_0$ where $e_1 : W_1 \xrightarrow{} \ker{\epsilon_A}$ is an epimorphism and $\eta_0 : \ker{\epsilon_A} \xrightarrow{} W_0$ is a monomorphism, we get a short exact sequence of chain complexes: 
$$\begin{tikzcd}
\vdots \arrow[d] & \vdots \arrow[d] \\
W_2 \arrow[r, equal] \arrow[d, "d_2" ']& W_2   \arrow[d, "d_2"] \\
W_1 \arrow[r, equal] \arrow[d, two heads, "e_1" '] & W_1  \arrow[r] \arrow[d, "d_1"] \arrow[ld, two heads, "e_1" ']  & 0 \arrow[d] \\
\ker{\epsilon_A} \arrow[r, tail, "\eta_0"]  \arrow[d, "0" '] & W_0  \arrow[r, ,two heads, "\epsilon_A"] \arrow[d, "d_0"] & A \arrow[d] \arrow[ld, tail, "\eta_A"] \\
W_{-1}  \arrow[r, equal]  \arrow[d, "d_{-1}" ']  &     W_{-1}    \arrow[r]  \arrow[d, "d_{-1}"]  &      0 \arrow[d]   \\
W_{-2}  \arrow[r, equal]   \arrow[d] &     W_{-2}    \arrow[r] \arrow[d] &      0    \\
\vdots & \vdots \\
\end{tikzcd}$$
Recall that the complex on the right is often denoted $S^0(A)$, and called the \emph{sphere on $A$}. With this notation we will denote the above short exact sequence by 
\begin{equation}\label{eq-ses}
0 \xrightarrow{} K_A \xrightarrow{} W_A \xrightarrow{\epsilon_A}  S^0(A) \xrightarrow{} 0.
\end{equation}

\begin{lemma}\label{lemma-ses}
Let $X \equiv (X_{n\geq0} \xrightarrow{\epsilon_B} B \xrightarrow{\eta_B} X_{n\leq-1})$ be a full augmentation of an object $B$ such that each $X_n$ is trivially cofibrant for $n\geq 0$, and  $B \xrightarrow{\eta_B} X_{n\leq-1}$ is a coresolution with $\cok{\eta_B}$ and each $\cok{d_n}$ cofibrant for $n\leq-1$. Then in  the short exact sequence of (\ref{eq-ses}), any chain map $\{f_n\} : X \xrightarrow{} K_A$ must be null homotopic. 
\end{lemma}

\begin{proof}
Looking at the construction of $K_A$ we note that $$K_A \equiv (K_{n\geq0} \xrightarrow{} 0 \xrightarrow{} K_{n\leq-1})$$ is a full augmentation of $0$,
where $K_{n\geq0} \xrightarrow{} 0$ is the resolution 
$$\cdots \xrightarrow{} W_2 \xrightarrow{d_2} W_1 \xrightarrow{e_1} \ker{\epsilon_A}  \xrightarrow{0} 0 \xrightarrow{} 0$$ which has every kernel (cycle) a fibrant object,  and, $0 \xrightarrow{} K_{n\leq-1}$ is the co-augmentation $$0 \xrightarrow{} 0 \xrightarrow{0} W_{-1} \xrightarrow{d_{-1}} W_{-2} \xrightarrow{} \cdots$$ which has all components trivially fibrant. 

Now any chain map $X \xrightarrow{} K_A$ is an extension of the trivial map $B \xrightarrow{} 0$. So such a chain map must be null homotopic by Lemma~\ref{lemma-full-augmentation-homotopies} by taking $\class{X} = \class{Q}_{\class{W}}$ and $\class{Y} = \class{R}_{\class{W}}$ in the statement of that lemma. 
\end{proof}

\begin{proposition}\label{prop-ses}
The short exact sequence (\ref{eq-ses}) satisfies the following: 
\begin{enumerate}
\item  Any chain map $W_B \xrightarrow{} K_A$ is null homotopic whenever $W_B$ is an $\mathfrak{M}$-resolution of some object $B$.  
\item  Any chain map $C \xrightarrow{} K_A$ is null homotopic whenever $C \in \class{C}$. 
\item $\Ext^1_{\cha{A}}(C,K_A) = 0$ for any $C \in \class{C}$. That is, $K_A \in \rightperp{\class{C}}$.
\item If $A$ is cofibrant, then the short exact sequence in (\ref{eq-ses}) is a special $\class{C}$-precover of $S^0(A)$.
\end{enumerate}
\end{proposition}

\begin{proof}
Statement (1) follows from Lemma~\ref{lemma-ses} by taking $X = W_B$.  Statement (2) also follows from Lemma~\ref{lemma-ses} by taking $X = C$ and $B = Z_{-1}C$. 

For (3), since $\ker{\epsilon_A} \in \class{R}$, we see that in fact all the components of $K_A$ are fibrant. Thus the Yoneda Ext group $\Ext^1_{\cha{A}}(C,K_A)$ coincides with the subgroup $\Ext^1_{dw}(C,K_A)$ of all degreewise split short exact sequences 
$0 \xrightarrow{} K_A \xrightarrow{} Z \xrightarrow{} C \xrightarrow{} 0$. But then from Example~\ref{example-K(A)}, it follows from (2) that any such short exact sequence truly splits in the category $\cha{A}$. That is, $\Ext^1_{\cha{A}}(C,K_A) = 0$.

For (4), note that if $A$ is cofibrant then the hereditary condition guarantees that $W_A \in \class{C}$.    So in this case, (3) tells us that $\epsilon_A : W_A \twoheadrightarrow S^0(A)$ is a special $\class{C}$-precover in the category $\cha{A}$. 
\end{proof}

\begin{theorem}\label{them-C-precovers}
For any object $A \in \cat{A}$, the complex $S^n(A)$ has a special $\class{C}$-precover. More specifically, the composition $W_{QA} \xrightarrow{\epsilon_{QA}} S^n(QA) \xrightarrow{S^n(p_A)} S^n(A)$ is a special $\class{C}$-precover whenever $QA \xrightarrow{p_A} A$ is a cofibrant replacement in $\mathfrak{M}$.
\end{theorem}

\begin{proof}
The composition is certainly an epimorphism, and we have $W_{QA} \in \class{C}$. So it is left to show that the kernel of the composition is in $\rightperp{\class{C}}$. Note that the diagram below is commutative and has exact rows.
$$\begin{CD}
     0   @>>> K_{QA}    @>>>  W_{QA}    @>\epsilon_{QA}>>  S^n(QA)    @>>>    0 \\
    @.        @VVV     @VS^n(p_A)\,\circ\,\epsilon_{QA}VV      @VVS^n(p_A)V   @.\\
    0   @>>>  0   @>>>   S^n(A)   @=  S^n(A)    @>>> 0   \\
\end{CD}$$ 
It follows from the snake lemma that the kernel of the composition $S^n(p_A)\,\circ\,\epsilon_{QA}$ is an extension of $K_{QA}$ and $\ker{S^n(p_A)}$. (For exact categories we  can apply the $3\times 3$-Lemma~\cite[Cor.3.6]{buhler-exact categories}.) Since $\rightperp{\class{C}}$ is closed under extensions and since $K_{QA} \in \rightperp{\class{C}}$ by Proposition~\ref{prop-ses}(4), it remains to show that $\ker{S^n(p_A)} \in \rightperp{\class{C}}$. But recall that a cofibrant replacement $p_A$ arises from a short exact sequence 
$$0 \xrightarrow{} R_A \xrightarrow{i_A} QA \xrightarrow{p_A} A \xrightarrow{} 0$$ with $QA \in \class{Q}$ and $R_A \in \class{R}_{\class{W}}$, so $R_A$  is trivially fibrant. Now for any $C \in \class{C}$, we see that $\Hom_{\cat{A}}(C,S^n(R_A))$ remains exact since the cycles of $C$ are all cofibrant. It follows that for all $C \in \class{C}$, any chain map $C \xrightarrow{} S^n(R_A)$ is null homotopic. But also $\Ext^1(C,S^n(R_A))$  consists of all \emph{degreewise split} short exact sequences $$0 \xrightarrow{} S^n(R_A) \xrightarrow{} Z \xrightarrow{} C \xrightarrow{} 0,$$ and it follows from this that $\Ext^1(C,S^n(R_A)) = 0$. Again we are using the statements in Example~\ref{example-K(A)}.
% OLD ARGUMENT: since $C$ is exact with cofibrant cycles we have, using an isomorphism from~\cite[Lemma~4.2]{gillespie-degreewise-model-strucs},  $\Ext^1_{\cha{A}}(C,S^n(R_A)) \cong \Ext^1_{\cat{A}}(Z_{n-1}C,R_A) = 0$. So $\ker{S^n(p_A)} = S^n(R_A)  \in \rightperp{\class{C}}$.
\end{proof}

Now assume that $\cat{A}$ is a Grothendieck category and that the model structure $\mathfrak{M} = (\class{Q}, \class{W}, \class{R})$ is cofibrantly generated. Becker shows in~\cite{becker-realization-functor} that $\mathfrak{M}$ lifts to a cofibrantly generated abelian model structure $(\class{C}, \class{V}, \dgclass{R})$, on $\cha{A}$. So the trivially cofibrant chain complexes are those in $\class{C} \cap \class{V} = \tilclass{Q_{\class{W}}}$, the class of all exact complexes with trivially cofibrant cycles. The class of all fibrant complexes is $\dgclass{R} := \rightperp{\tilclass{Q_{\class{W}}}}$. The next result tells us that cofibrant replacements in this model structure can be used to compute $\Ext^n_{\textnormal{Ho}(\mathfrak{M})}(A,B)$. %and $\Tor_n^{\textnormal{Ho}(\mathfrak{M})}(A,B)$. 
Dual results hold for special $\class{F}$-preenvelopes and fibrant approximations.

\begin{corollary}\label{cor-Becker}
Assume that $\cat{A}$ is a Grothendieck category and that $\mathfrak{M} = (\class{Q}, \class{W}, \class{R})$ is cofibrantly generated. Then any canonical resolution of $A \in \cat{A}$ provides a cofibrant replacement of $S^0(A)$ in the model structure $(\class{C}, \class{V}, \dgclass{R})$. Moreover,
$$\Ext^n_{\textnormal{Ho}(\mathfrak{M})}(A,B) \cong H^n[\Hom_\cat{A}(C,RB)]$$
where $C$ is any cofibrant replacement of $S^0(A)$ in the model structure $(\class{C}, \class{V}, \dgclass{R})$, and $RB$ is any fibrant replacement of $B$ in $\mathfrak{M}$. 
\end{corollary}

\begin{proof}
The special $\class{C}$-precover of $S^0(A)$ from Theorem~\ref{them-C-precovers} provides a short exact sequence $0 \xrightarrow{} J  \xrightarrow{} W_{QA} \xrightarrow{}  S^0(A) \xrightarrow{} 0$. It has $W_{QA} \in \class{C}$ and $J \in \rightperp{\class{C}}$. This is precisely a cofibrant replacement of $S^0(A)$, in $(\class{C}, \class{V}, \dgclass{R})$. Given any other such cofibrant replacement sequence  $0 \xrightarrow{} K \xrightarrow{} C \xrightarrow{}  S^0(A) \xrightarrow{} 0$ we deduce from~\cite[Fundamental Lemma~3.4]{gillespie-stable-cats-cotorsion-pairs} that there is a canonical isomorphism $[f]_{\class{X}} : W_{QA} \xrightarrow{} C$ in the category $\textnormal{St}_{\class{X}}(\cha{A})$, where $\class{X} = \class{C} \cap \rightperp{\class{C}}$. This is the stable category of $\cha{A}$ modulo the equivalence relation $\sim^{\class{X}}$ defined by $f \sim^{\class{X}} g$ iff $g-f$ factors through an object of $\class{X}$. But one can verify that the complexes in $\class{X}$ are precisely the contractible complexes with components in the core $\omega = \class{Q} \cap \class{W} \cap \class{R}$. Since these complexes are contractible,  $f \sim^{\class{X}} g$ implies $f$ and $g$ are chain homotopic. It follows that a representative $f$ of the canonical isomorphism $[f]_{\class{X}} : W_{QA} \xrightarrow{} C$ must be a chain homotopy equivalence. It then follows that the induced morphisms
$$H^n[\Hom_{\class{A}}(C, RB)]\xrightarrow{H^n([f^*])} H^n[\Hom_{\class{A}}(W_{QA}, RB)]$$  are all isomorphisms. So the result now follows from Theorem~\ref{them-Ho-Ext}(3). 
\end{proof}

\subsection{Special $\class{C}$-precovers of trivial resolutions}
Again let $A \in \cat{A}$, but now consider a full trivial resolution  $W \equiv (W_{n\geq 0} \xrightarrow{\epsilon}A\xrightarrow{\eta}W_{n\leq -1})$ in the sense of Definition~\ref{def-full-trivial-resolutions}. The proof of Proposition~\ref{prop-functors-commute} constructs a short exact sequence 
\begin{equation}\label{eq-ses2}
0 \xrightarrow{} R \xrightarrow{} QW \xrightarrow{}  W \xrightarrow{} 0,
\end{equation}
where $R \in \tilclass{R_{\class{W}}}$, the class of all exact complexes with trivially fibrant cycles, and $QW \in \class{C}$. 

\begin{theorem}
The epimorphism in the short exact sequence of~(\ref{eq-ses2}) is a special $\class{C}$-precover and we have an isomorphism
$$\Ext^n_{\textnormal{Ho}(\mathfrak{M})}(A,B) \cong H^n[\Hom_\cat{A}(QW,RB)].$$ 
\end{theorem}

\begin{proof}
It follows from~\cite[Lemma~2.1 and Lemma~3.9]{gillespie} that $\Ext^1_{\cha{A}}(C,R) = 0$ for all $C \in \class{C}$. Therefore the epimorphism $QW \xrightarrow{} W$ of~(\ref{eq-ses2}) is a special $\class{C}$-precover of $W$.

Next, the argument in the proof of Theorem~\ref{them-lExt and rExt} can be imitated to show 
$$H^n[\Hom_\cat{A}(QW,RB)] \cong \underline{\Hom}_{\,\omega}(Q\Omega^n A, RB) \cong  \textnormal{Ho}(\mathfrak{M})(\Omega^n A,B) $$
and so the result follows from Corollary~\ref{cor-Ho-Ext}.
\end{proof}

Now we get the following corollary in the same way we obtained Corollary~\ref{cor-Becker}.

\begin{corollary}
Assume that $\cat{A}$ is a Grothendieck category and that $\mathfrak{M} = (\class{Q}, \class{W}, \class{R})$ is cofibrantly generated. Then 
$$\Ext^n_{\textnormal{Ho}(\mathfrak{M})}(A,B) \cong H^n[\Hom_\cat{A}(QW,RB)]$$
where $QW$ is any cofibrant replacement, in the model structure $(\class{C}, \class{V}, \dgclass{R})$, of a full trivial resolution $W \equiv (W_{n\geq 0} \xrightarrow{\epsilon}A\xrightarrow{\eta}W_{n\leq -1})$, and $RB$ is any fibrant replacement of $B$ in $\mathfrak{M}$. 
\end{corollary}

\appendix

\section{The suspension and loop functor}\label{sec-suspension and loop}

In this Appendix we give a direct construction of the suspension and loop functors on the homotopy category of an hereditary exact model structure.  So throughout, we let $\cat{A}$ be an exact category and assume it to have an exact model structure and we let $\mathfrak{M} = (\class{Q},\class{W}, \class{R})$ denote the associated triple~\cite[Theorem~3.3]{gillespie-exact model structures}. We also continue to assume that the cotorsion pairs are hereditary and we let $\omega := \class{Q}\cap\class{W}\cap \class{R}$ denote its \emph{core}. 
For those interested only in abelian categories, simply replace ``admissible monomorphism'' with `` monomorphism'' and ``admissible epimorphism'' with ``epimorphism''. See Section~\ref{sec-exact-cats} and~\cite{buhler-exact categories} and~\cite{gillespie-exact model structures}.

We gave a definition of $\textnormal{Ho}(\mathfrak{M})$, the homotopy category of $\mathfrak{M}$,  in Section~\ref{sec-prelims}. Using this definition we will now give a direct construction of $\textnormal{Ho}(\mathfrak{M}) \xrightarrow{\Sigma} \textnormal{Ho}(\mathfrak{M})$, the \emph{suspension functor}, and its inverse $\textnormal{Ho}(\mathfrak{M}) \xrightarrow{\Omega} \textnormal{Ho}(\mathfrak{M})$, the \emph{loop functor}.
In an appendix of his thesis, Hanno Becker gave a different proof that showed that the loop and suspension functors can be computed as described in the next proposition. 

%Like much of the theory of abelian model categories, one may often relax the category $\cat{A}$ from being abelian to being a weakly idempotent complete additive category with an exact structure in the sense of~\cite{buhler-exact categories} and~\cite{gillespie-exact model structures}. This Appendix is written in such generality. 

\begin{proposition}\label{prop-suspension functor}
Let $\mathfrak{M} = (\class{Q},\class{W}, \class{R})$ be an hereditary exact model structure on $\class{A}$ and $\textnormal{Ho}(\mathfrak{M})$ its homotopy category.
\begin{enumerate}
\item The process of taking the cokernel of any admissible monomorphism  $X \rightarrowtail W$ with $W \in \class{W}$ determines an additive endofunctor $\textnormal{Ho}(\mathfrak{M}) \xrightarrow{\Sigma} \textnormal{Ho}(\mathfrak{M})$, called the \textbf{suspension functor}. It is well-defined on objects up to a canonical isomorphism. 
\item The process of taking the kernel of any admissible epimorphism  $W  \twoheadrightarrow X$ with $W \in \class{W}$ determines an additive endofunctor $\textnormal{Ho}(\mathfrak{M}) \xrightarrow{\Omega} \textnormal{Ho}(\mathfrak{M})$, called the \textbf{loop functor}. It is well-defined on objects up to a canonical isomorphism.
\item $\Sigma$ and $\Omega$ are each left and right adjoint to each other. 
\item In fact, $\Sigma$ and $\Omega$ are inverse automorphisms of $\textnormal{Ho}(\mathfrak{M})$. In particular, $\Sigma$ is what is usually called a \emph{translation functor on $\textnormal{Ho}(\mathfrak{M})$}.
\end{enumerate}
\end{proposition}

\begin{proof}
We first note that since the cotorsion pairs are hereditary we may use the generalized horseshoe lemma, \cite[Lemma~1.4.4]{becker} or see~\cite{gillespie-stable-cats-cotorsion-pairs},  to construct, for any given admissible short exact sequence $X \rightarrowtail W  \twoheadrightarrow Y$ with $W \in \class{W}$,  a commutative diagram as below whose rows are  admissible short exact sequences, and top row a bifibrant approximation sequence for $X \rightarrowtail W  \twoheadrightarrow Y$. 
$$  \begin{CD}
        0 @>>>   RQX    @>>>   RQW    @>>>    RQY    @>>>    0 \\
    @.        @AAA     @AAA     @AAA   @.\\
    0   @>>>  QX    @>>>   QW    @>>>  QY    @>>>  0  \\
    @.        @VVV     @VVV      @VVV   @.\\
    0   @>>>  X    @>>>   W    @>>>    Y    @>>>  0  \\
    \end{CD}$$ 
Note that $RQW \in \omega$, since $\class{W}$ is closed under extensions. 

Now consider two arbitrary objects $X$ and $X'$ and two arbitrary admissible short exact sequences $X \rightarrowtail W  \twoheadrightarrow Y$, and $X' \rightarrowtail W'  \twoheadrightarrow Y'$, with $W,  W' \in \class{W}$. A morphism $X \xrightarrow{} X'$,  in $\textnormal{Ho}(\mathfrak{M})$,  is (by definition) a morphism in $\underline{\Hom}_{\,\omega}(RQX,RQX')$; so it is an  equivalence class $RQX \xrightarrow{[f]} RQX'$ for the relation $\sim^{\omega}$. But $(\class{Q}\cap\class{R},\omega)$ is a complete cotorsion pair in the Frobenius category $\class{Q}\cap\class{R}$, and so by \cite[Dual Fundamental Lemma~2.5]{gillespie-stable-cats-cotorsion-pairs}, the representative map $RQX \xrightarrow{f} RQX'$ completes to a morphism of short exact sequences
$$  \begin{CD}
      0   @>>>    RQX    @>>>   RQW    @>>>  RQY    @>>>    0 \\
    @.        @V f VV     @V h VV      @V c VV   @.\\
     0   @>>>  RQX'    @>>>   RQW'    @>>>  RQY'    @>>> 0.    \\
    \end{CD}$$ Moreover, $[c] \in \underline{\Hom}_{\,\omega}(RQY,RQY')$ is uniquely determined by just $[f]$, the equivalence class of $f$.

Thus by taking $X = X'$ and $[f] = [1_X]$,  we get a morphism $QY \xrightarrow{[c]} QY'$ and a reversal  $QY' \xrightarrow{[d]} QY$ which together must satisfy $[d] \circ [c] = [1_{QY}]$ and  $[c] \circ [d] = [1_{QY'}]$ (due to the uniqueness portion of \cite[Dual Fundamental Lemma~2.5]{gillespie-stable-cats-cotorsion-pairs}). Of course $[c]$ is a morphism from $Y\xrightarrow{} Y'$, in $\textnormal{Ho}(\mathfrak{M})$, and so it is an isomorphism with inverse $[d]$. Thus we have shown that the process of taking a $\Sigma X$ sitting within an admissible short exact sequence $X \rightarrowtail W  \twoheadrightarrow \Sigma X$,  with $W \in \class{W}$, is unique up to a canonical isomorphism in $\textnormal{Ho}(\mathfrak{M})$.

But in fact the action $X \mapsto \Sigma X$ is not just well-defined on objects: One verifies with similar arguments that the general action $[f] \mapsto \Sigma([f]) :=  [c]$ determines a functor, $\textnormal{Ho}(\mathfrak{M}) \xrightarrow{\Sigma} \textnormal{Ho}(\mathfrak{M})$. The uniqueness portion of \cite[Dual Fundamental Lemma~2.5]{gillespie-stable-cats-cotorsion-pairs} will also show that $\Sigma$ is additive. 
On the other hand, similar arguments using \cite[Fundamental Lemma~2.4]{gillespie-stable-cats-cotorsion-pairs}
provide us with the dual loop functor $\textnormal{Ho}(\mathfrak{M}) \xrightarrow{\Omega} \textnormal{Ho}(\mathfrak{M})$. 

We now show $\Omega$ is left adjoint to $\Sigma$. Recall that this means there is a bijection of the hom-sets $\delta_{X,Y} \mathcolon \textnormal{Ho}(\mathfrak{M})(\Omega X,Y)  \cong  \textnormal{Ho}(\mathfrak{M})(X,\Sigma Y)$, that is, $$\delta_{X,Y} \mathcolon \underline{\Hom}_{\,\omega}(RQ\Omega X,RQY) \cong  \underline{\Hom}_{\,\omega}(RQX,RQ\Sigma Y),$$ and that it is natural in both $X$ and $Y$. But it is easy now to define such a natural bijection. Given $X$ and $Y$, write the short exact sequences $0 \xrightarrow{} \Omega X \xrightarrow{} W_X \xrightarrow{} X \xrightarrow{} 0$ and $0 \xrightarrow{} Y \xrightarrow{} W_Y \xrightarrow{} \Sigma Y \xrightarrow{} 0$. Then for $[f] \in \underline{\Hom}_{\,\omega}(RQ\Omega X,RQY)$, define $\delta_{X,Y}([f]) = [c]$ in the diagram below where we again are using \cite[Dual Fundamental Lemma~2.5]{gillespie-stable-cats-cotorsion-pairs} applied to the cotorsion pair $(\class{Q}\cap\class{R},\omega)$ in the Frobenius category $\class{Q}\cap\class{R}$:
$$\begin{CD}
     0   @>>>  RQ\Omega X    @>>>   RQW_X    @>>>   RQX    @>>>    0 \\
    @.        @V f VV     @VVV      @V c VV   @.\\
    0   @>>>  RQY    @>>>   RQW_Y    @>>>   \Sigma RQY    @>>> 0   \\
    \end{CD}
$$ Then it is clear that $\delta$ is a bijection with inverse defined by $\delta_{X,Y}^{-1}([c]) = [f]$ because of the Fundamental Lemmas. Naturality in $Y$ means that given any morphism  $$([g] \mathcolon Y_1 \xrightarrow{} Y_2) \in  \textnormal{Ho}(\mathfrak{M})(Y_1,Y_2) = \underline{\Hom}_{\,\omega}(RQY_1,RQY_2),$$ we have commutativity of the square:
$$\begin{CD}
     \underline{\Hom}_{\,\omega}(RQ\Omega X,RQY_1)    @> \delta_{X,Y_1} >>  \underline{\Hom}_{\,\omega}(RQX, RQ\Sigma Y_1) \\
            @V [g]_* VV     @VV (\Sigma [g])_* V \\
    \underline{\Hom}_{\,\omega}(RQ\Omega X,RQY_2)    @>> \delta_{X,Y_2} >  \underline{\Hom}_{\,\omega}(RQX, RQ\Sigma Y_2)    \\
\end{CD}$$
So given $ [f] \in \underline{\Hom}_{\,\omega}(RQ\Omega X,RQY_1)$ we form the commutative diagram:
$$\begin{CD}
     0   @>>> RQ\Omega X    @>>>   RQW_X    @>>>    RQX    @>>>    0 \\
    @.        @V f VV     @VVV      @VV c V   @.\\
     0   @>>> RQY_1    @>>>   RQW_{Y_1}    @>>>   RQ\Sigma Y_1    @>>>    0 \\
    @.        @V g VV     @VVV      @VV c' V   @.\\
    0   @>>>   RQY_2    @>>>   RQW_{Y_2}    @>>>  RQ\Sigma Y_2    @>>> 0   \\
\end{CD}$$
In order to check commutativity of the naturality diagram, we compute $$\delta_{X,Y_2}([g]_*([f])) = \delta_{X,Y_2}([g]\circ[f]) = \delta_{X,Y_2}([gf]).$$ And on the other hand, $$(\Sigma [g])_*(\delta_{X,Y_1}([f])) = (\Sigma [g]) \circ ([c]) = [c'] \circ [c] = [c'c].$$ But $\delta_{X,Y_2}([gf]) = [c'c]$ by \cite[Dual Fundamental Lemma~2.5]{gillespie-stable-cats-cotorsion-pairs}. This completes the verification of naturality in $Y$. To check naturality in $X$ one must show that given $[h] \mathcolon X_1 \xrightarrow{} X_2$, we have commutativity of the square:
$$\begin{CD}
     \underline{\Hom}_{\,\omega}(RQ\Omega X_2,RQY)    @> \delta_{X_2,Y} >>  \underline{\Hom}_{\,\omega}(RQX_2, RQ\Sigma Y) \\
            @V \Omega([h])^* VV     @VV [h]^* V \\
    \underline{\Hom}_{\,\omega}(RQ\Omega X_1, RQY)    @>> \delta_{X_2,Y} >  \underline{\Hom}_{\,\omega}(RQX_1, RQ\Sigma Y) \ .   \\
\end{CD}$$ This holds in a similar way and proves that $\Omega$ is left adjoint to $\Sigma$. Moreover, the same methods will show that $\Omega$ is also right adjoint to $\Sigma$.

\

We now show that there are natural isomorphisms of functors $\Omega \Sigma \cong I_{\textnormal{Ho}(\mathfrak{M})}$ and $\Sigma \Omega \cong I_{\textnormal{Ho}(\mathfrak{M})}$, where $I_{\textnormal{Ho}(\mathfrak{M})}$ is the identity functor on $\textnormal{Ho}(\mathfrak{M})$. We just do this for $\Sigma \Omega \cong I_{\textnormal{Ho}(\mathfrak{M})}$. Thus we are required to construct for each $X \in \textnormal{Ho}(\mathfrak{M})$, an isomorphism $([\psi_X] \mathcolon \Sigma \Omega X \xrightarrow{} X) \in \underline{\Hom}_{\,\omega}(RQ\Sigma \Omega X, RQX)$ such that given any $([f] \mathcolon X \xrightarrow{} Y) \in \underline{\Hom}_{\,\omega}(RQX,RQY)$ we have a commutative square:
$$\begin{CD}
     \Sigma \Omega X    @> [\psi_X] >>  X \\
            @V \Sigma \Omega ([f]) VV     @VV [f] V \\
    \Sigma \Omega Y    @>> [\psi_Y] >  Y    \\
\end{CD}$$

Applying the definition of $\Omega$, we obtain a commutative diagram
$$\begin{CD}
     0   @>>> RQ\Omega X    @>>>  RQW_X    @>>>  RQX    @>>>    0 \\
    @.        @V k VV     @V h VV      @V f VV   @.\\
    0   @>>>  RQ\Omega Y    @>>>   RQW_Y    @>>>  RQY    @>>> 0   \\
\end{CD}$$ and $\Omega ([f]) = [k]$. Then from the definition of $\Sigma$ we get short exact sequences $0 \xrightarrow{} \Omega X \xrightarrow{} W_{\Omega X} \xrightarrow{} \Sigma \Omega X \xrightarrow{} 0$ and $0 \xrightarrow{} \Omega Y \xrightarrow{} W_{\Omega Y} \xrightarrow{} \Sigma \Omega Y \xrightarrow{} 0$ and these lead us to the commutative diagram (three sides of a parallelepiped):
$$\begin{CD}
     0   @>>> RQ\Omega X    @>>>   RQW_{\Omega X}    @>>>    RQ\Sigma \Omega X    @>>>    0 \\
    @.        @|     @V h_1 VV      @VV \psi_X V   @.\\
     0   @>>> RQ\Omega X    @>>>   RQW_X    @>>>    RQX    @>>>    0 \\
    @.        @V k VV     @V h VV      @V f VV   @.\\
    0   @>>>  RQ\Omega Y    @>>>   RQW_Y    @>>>  RQY    @>>> 0   \\
    @.        @|     @A h_2 AA      @AA \psi_Y A   @.\\
    0   @>>>  RQ\Omega Y    @>>>   RQW_{\Omega Y}   @>>> RQ\Sigma \Omega Y    @>>> 0   \\
\end{CD}$$
Since we can find reversals for $\psi_X$ and $\psi_Y$, it readily follows from \cite[Dual Fundamental Lemma~2.5]{gillespie-stable-cats-cotorsion-pairs} that $[\psi_X] \mathcolon \Sigma \Omega X \xrightarrow{} X$ and $[\psi_Y] \mathcolon \Sigma \Omega Y \xrightarrow{} Y$ are isomorphisms. So these form the natural isomorphisms that we seek, we just need to finish the argument that the proper square commutes.

But we may compose the left most vertical arrows, and this composition is just $k \mathcolon RQ\Omega X \xrightarrow{} RQ\Omega Y$. Applying the definition of $\Sigma$ we get fill in maps $h' \mathcolon  RQW_{\Omega X} \xrightarrow{}  RQW_{\Omega Y}$ and $c \mathcolon RQ\Sigma \Omega X \xrightarrow{} RQ\Sigma \Omega Y$ so that $\Sigma([k]) = [c]$. Furthermore, \cite[Dual Fundamental Lemma~2.5]{gillespie-stable-cats-cotorsion-pairs} implies $[h h_1] = [h_2 h']$ and $[f \psi_X] = [\psi_Y c]$. In particular, $[f] \circ [\psi_X] = [\psi_Y] \circ [c]$. That is,  $[f] \circ [\psi_X] = [\psi_Y] \circ \Sigma \Omega([f])$.
\end{proof}

\begin{example}\label{example-suspensions} 
Let $\cat{A}$ be any exact category. Let $\cha{A}$ denote the associated chain complex category along with the exact structure it inherits degreewise. Suppose $\mathfrak{M} = (\class{Q},\class{W},\class{R})$ is an hereditary exact model structure on $\cha{A}$. If $\class{W}$ contains all contractible complexes, then we can take $\Sigma$ to be the usual suspension functor (shift against arrows, and change differentials to $-d$), and $\Omega$ to be its inverse $\Sigma^{-1}$. (Reason.) It is an easy exercise to construct degreewise split short exact sequences 
$$ 0 \xrightarrow{} X \xrightarrow{} \bigoplus_{n \in \Z} D^{n+1}(X_n) \xrightarrow{} \Sigma X \xrightarrow{} 0$$ 
and
$$ 0 \xrightarrow{} \Sigma^{-1} X \xrightarrow{} \bigoplus_{n \in \Z} D^{n}(X_n) \xrightarrow{} X \xrightarrow{} 0$$ 
and $\bigoplus_{n \in \Z} D^{n+1}(X_n)$ and $\bigoplus_{n \in \Z} D^{n}(X_n)$ are contractible complexes. 
\end{example}

% \bibliography{hovey}
% \bibliographystyle{amsalpha}

\providecommand{\bysame}{\leavevmode\hbox to3em{\hrulefill}\thinspace}
\providecommand{\MR}{\relax\ifhmode\unskip\space\fi MR }
% \MRhref is called by the amsart/book/proc definition of \MR.
\providecommand{\MRhref}[2]{%
  \href{http://www.ams.org/mathscinet-getitem?mr=#1}{#2}
}
\providecommand{\href}[2]{#2}

\end{document}